\documentclass[leqno]{amsart} 
                       
\usepackage{pstricks,color}

\usepackage{pst-plot}

\usepackage{amsmath,amsthm}     
\usepackage{amssymb}            
\usepackage{euscript}           
\usepackage{graphicx,enumerate,calc,lscape,color}
\usepackage[matrix,arrow,curve,frame]{xy}   

\xymatrixcolsep{1.9pc}                          
\xymatrixrowsep{1.9pc}
\newdir{ >}{{}*!/-5pt/\dir{>}}                  

\newtheorem{thm}[subsection]{Theorem}
\newtheorem{defn}[subsection]{Definition}
\newtheorem{prop}[subsection]{Proposition}

\newtheorem{cor}[subsection]{Corollary}
\newtheorem{lemma}[subsection]{Lemma}

\theoremstyle{definition}  
\newtheorem{ex}[subsection]{Example}

\newtheorem{remark}[subsection]{Remark}

\newcommand{\tens}              {\otimes}               
\newcommand{\iso}               {\cong}

\newcommand{\M}         {\mathbb{M}}

\newcommand{\F}		    {\mathbb{F}}
\newcommand{\C}         {\mathbb{C}}
\newcommand{\Z}         {\mathbb{Z}}
\newcommand{\R}         {\mathbb{R}}

\DeclareMathOperator{\Spec}{Spec}
\DeclareMathOperator{\Hom}{Hom}

\DeclareMathOperator{\Ext}{Ext}

\DeclareMathOperator{\Gr}{Gr}

\DeclareMathOperator{\Sq}{Sq}
\DeclareMathOperator{\Ch}{Ch}

\newcommand{\map}{\rightarrow}

\newcommand{\anglb}[1]{{{\langle}#1{\rangle}}}

\newcommand{\cl}{\mathrm{cl}}

\numberwithin{equation}{subsection}

\definecolor{red}{rgb}{1,0,0}
\definecolor{blue}{rgb}{0.3,0.3,1}

\newrgbcolor{tauzerocolor}{0.5 0.5 0.5}
\newrgbcolor{tauonecolor}{1 0 0}
\newrgbcolor{tautwocolor}{0.3 0.3 1}
\newrgbcolor{tauthreecolor}{0 1 0}

\newrgbcolor{hzerocolor}{0.5 0.5 0.5}
\newrgbcolor{hzerotaucolor}{1 0 1}
\newrgbcolor{hzeromoretaucolor}{1 0.5 0}
\newrgbcolor{hzerotowercolor}{0.5 0.5 0.5}

\newrgbcolor{honecolor}{0.5 0.5 0.5}
\newrgbcolor{honetaucolor}{1 0 1}
\newrgbcolor{honemoretaucolor}{1 0.5 0}
\newrgbcolor{honetowercolor}{1 0 0}

\newrgbcolor{htwocolor}{0.5 0.5 0.5}
\newrgbcolor{htwotaucolor}{1 0 1}
\newrgbcolor{htwomoretaucolor}{1 0.5 0}

\newrgbcolor{donecolor}{0 0.7 0.7}
\newrgbcolor{dtwocolor}{0 0.7 0.7}
\newrgbcolor{dthreecolor}{0 0.7 0.7}
\newrgbcolor{dfourcolor}{0 0.7 0.7}
\newrgbcolor{dfivecolor}{0.2 0.2 0.7}

\newgray{gridline}{0.8}

\newpsobject{hzerotower}{psline}{linecolor=hzerotowercolor}

\newpsobject{honetower}{psline}{linecolor=honetowercolor}

\newpsobject{done}{psline}{linecolor=donecolor}
\newpsobject{dtwo}{psline}{linecolor=dtwocolor}
\newpsobject{dthree}{psline}{linecolor=dthreecolor}
\newpsobject{dfour}{psline}{linecolor=dtwocolor}
\newpsobject{dfive}{psline}{linecolor=dfivecolor}

\newpsobject{taumult}{psline}{linestyle=dashed,dash=0.1 0.2}

\newrgbcolor{Pmult}{1 0 0}
\newrgbcolor{gmult}{0 0.7 0}
\newrgbcolor{multcolor}{0 0 1}
\newrgbcolor{mmfext}{1 0 0}

\newpsobject{extn}{psline}{linecolor=mmfext}
\newpsobject{mult}{psline}{linecolor=multcolor}
\newrgbcolor{gsquare}{0 0 1}
\newrgbcolor{gcube}{1 0 1}

\newcommand{\cirrad}{0.12}

\begin{document}

\title{The cohomology of the motivic Steenrod algebra over $\Spec \C$}
\author{Daniel C.\ Isaksen}

\begin{abstract}
The purpose of this article
is to compute 
the cohomology of the motivic Steenrod algebra over $\Spec \C$
through the geometric 70-stem.
The main computational tool is the motivic May spectral sequence. 
Everywhere in this article, we are working only over $\Spec \C$,
and we are always computing at the prime $p = 2$.
\end{abstract}

\maketitle


\section{Introduction}

This article is one piece of a long-term project to explore
stable homotopy groups in motivic homotopy theory.  
Just as in classical homotopy theory, the stable motivic homotopy groups
are key to understanding the behavior of the stable cellular 
motivic homotopy
category.  
The goal of this long-term project is to supply a large quantity of
computational data about motivic stable homotopy rings.  This data
can be used to detect large-scale patterns in motivic stable
homotopy that do not arise in classical stable homotopy.
Even better, the richer motivic stable homotopy setting allows
one to easily detect some otherwise difficult
facts about the classical situation.
We provide some examples of these payoffs 
in Section \ref{sctn:examples}.

One important difference between the classical case and the
motivic case is that not every motivic spectrum is built out of spheres,
i.e., not every motivic spectrum is cellular.
Stable cellular motivic homotopy theory is more tractable than the
full motivic homotopy theory, and 
many motivic spectra of particular interest, such as
the Eilenberg-Mac Lane spectrum $H\F_2$, 
the algebraic $K$-theory spectrum $KGL$, 
and the algebraic cobordism spectrum $MGL$, 
are all cellular.  
On the other hand, some arithmetic and algebro-geometric 
questions are not accessible via the cellular theory.

Although one can study motivic homotopy theory over any base field
(or even more general base schemes), we will work only over 
$\Spec \C$, or any algebraically closed field of characteristic $0$.
The point is to explore the parts of motivic homotopy theory that
don't depend on the subtle arithmetic of the base field.
Even in this simplest case, we can find exotic phenomena that have
no classical analogues.

One way to approach the stable motivic homotopy groups is through
the motivic Adams spectral sequence.
First we must choose a prime $p$.  In this article, we are always
working at $p = 2$.
In order to carry out the $2$-primary motivic Adams spectral sequence,
one must first obtain the $E_2$-term,
which is the cohomology of the motivic Steenrod algebra.
The purpose of this article
is to describe this $E_2$-term explicitly through 
the geometric 70-stem.
We will use the motivic May spectral sequence to carry out the
computation. 

After obtaining the $E_2$-term, the next step is to analyze
differentials in the motivic Adams spectral sequence.
This will appear in a sequel \cite{I2} to this article.

A second purpose of this calculation is to provide input to the
$\rho$-Bockstein spectral sequence for computing the 
cohomology of the motivic Steenrod algebra over $\Spec \R$.
The analysis of the $\rho$-Bockstein spectral sequence,
and the further analysis of the motivic Adams spectral sequence
over $\Spec \R$, will appear in \cite{DHI}.

The exposition of such a technical calculation creates some
inherent challenges.
The goal of the article is not to describe every step from first
principles.  Rather, it is a guide for reproducing the computation
for those who are already well-versed in computations with the classical
May spectral sequence \cite{Ma1} \cite{T}.
We give the key facts and examples of how the arguments work.

This article is a natural sequel to \cite{DI}, where
the first computational properties of the motivic May spectral sequence, 
as well as of the motivic Adams spectral sequence, were established.
We will review the key inputs from \cite{DI}.

We have been careful to adhere to a philosophy of logical dependence.
We will freely borrow results from the classical
May spectral sequence, i.e., from  \cite{Ma1} and \cite{T},
without further explanation.  The point is that this article
primarily focuses on the phenomena that are truly new in the motivic
context.

We will also need
some facts from the cohomology of the classical Steenrod
algebra that have been verified only by machine \cite{Br1} \cite{Br2}.
The machine computations construct a minimal resolution of $\F_2$
over the classical Steenrod algebra.  From this resolution,
one can derive the full structure of 
the cohomology of the classical Steenrod algebra,
including multiplicative structures, Massey products, and squaring
operations.

\subsection{Outline}

We begin in Section \ref{sctn:review} with a review of the
basic facts about the motivic Steenrod algebra over $\Spec \C$,
as well as the motivic May spectral sequence over $\Spec \C$.
We also establish the computational relationship between the
motivic setting and the classical setting.
This section is mostly a summary of some parts of \cite{DI}.

Next, in Section \ref{sctn:May} we describe the main points in
computing the motivic May spectral sequence through the geometric
70-stem.  We rely heavily on results of \cite{Ma1} and \cite{T}, 
but we must also
compute several exotic differentials, i.e., 
differentials that do not occur in the classical situation.

Having obtained the $E_\infty$-term of the motivic May spectral sequence,
the next step is to consider hidden extensions.
In Section \ref{sctn:hidden-extn},
we are able to resolve every possible hidden extension by
$\tau$, $h_0$, $h_1$, and $h_2$ through the range that we are considering,
i.e., up to the geometric 70-stem.
The primary tools here are:
\begin{enumerate}
\item
shuffling relations among Massey products.
\item
squaring operations on $\Ext$ groups in the sense of \cite{Ma2}.
\item
classical hidden extensions established by machine computation \cite{Br1}.
\end{enumerate}

In Section \ref{sctn:h1-local}, we consider what happens
when the element $h_1$ is inverted in the cohomology of 
the motivic Steenrod algebra.  Because $h_1^k$ is non-zero for all $k$,
this becomes an interesting calculation.

Section \ref{sctn:table} consists of a series of tables that are essential
for bookkeeping throughout the computations.  The most important of these
is Table 14, which lists the multiplicative generators of the 
cohomology of the motivic Steenrod algebra over $\Spec \C$.

Finally, in Section \ref{sctn:Ext},
we provide charts that describe 
the cohomology of the motivic Steenrod algebra over $\Spec \C$,
through the geometric 70-stem.
These charts are the main contribution of this article.
The diagrams are best viewed in color, although they are legible in
black and white.  Large format contiguous charts are available as well.

\subsection{Some examples}
\label{sctn:examples}

In this section, we describe several of the computational 
intricacies that are established later in the article.  
We also present a few questions deserving of further study.
On a first reading, this section may be safely skipped and 
reviewed later.

\begin{ex}
\label{ex:h2g^2}
An obvious question, which already arose in \cite{DI}, is to find
elements that are killed by $\tau^n$ but not by $\tau^{n-1}$,
for various values of $n$.

The element $h_2 g^2$, which is multiplicatively indecomposable,
is the first example
of an element that is killed by $\tau^3$ but not by $\tau^2$.
This occurs because of a hidden extension
$\tau \cdot \tau h_2 g^2 = P h_1^4 h_5$.
There is an analogous relation $\tau^2 h_2 g = P h_4$,
which occurs already in the May spectral sequence.
We do not know if this generalizes to a family of relations of the
form $\tau^2 h_2 g^{2^k} = P h_1^{2^{k+2}-4} h_{k+4}$.

We will show in \cite{I2} that $h_2 g^2$ represents an element
in motivic stable homotopy that is killed by $\tau^3$ but not by
$\tau^2$.  This requires an analysis of the motivic Adams spectral
sequence.
In the vicinity of $g^{2^k}$, one might hope to find elements that 
are killed by $\tau^n$ but not by $\tau^{n-1}$,
for large values of $n$.
\end{ex}

\begin{ex}
\label{ex:h3e0}
Classically, there is a relation $h_3 \cdot e_0 = h_1 h_4 c_0$
in the geometric 24-stem 
of the cohomology of the Steenrod algebra.  This relation is hidden on the
$E_\infty$-page of the May spectral sequence.
We now give a proof of this classical relation that uses the cohomology
of the motivic Steenrod algebra.  It still relies on a hidden
extension, but in a more elementary way.

Motivically, it turns out that $h_2^3 e_0$ is non-zero,
even though it is zero classically.  
This follows from the hidden extension $h_0 \cdot h_2^2 g = h_1^3 h_4 c_0$
(see Lemma \ref{lem:hidden-h0h2^2g}).
The relation $h_2^3 = h_1^2 h_3$ then implies that
$h_1^2 h_3 e_0$ is non-zero.  Therefore,
$h_3 e_0$ is non-zero as well, and the only possibility is that
$h_3 e_0 = h_1 h_4 c_0$.
%
\end{ex}

\begin{ex}
\label{ex:h1^7h5c0}
Notice the hidden extension
$h_0 \cdot h_2^2 g^2 = h_1^7 h_5 c_0$ (and similarly,
the hidden extension $h_0 \cdot h_2^2 g = h_1^3 h_4 c_0$
that we discussed above in Example \ref{ex:h3e0}).

The next example in this family is
$h_0 \cdot h_2^2 g^3 = h_1^9 D_4$, which at first does not
appear to fit a pattern.  However,
there is a hidden extension $c_0 \cdot i_1 = h_1^4 D_4$,
so we have
$h_0 \cdot h_2^2 g^3 = h_1^5 c_0 i_1$.
Presumably, 
there is an infinitely family of hidden extensions
in which $h_0 \cdot h_2^2 g^k$ equals 
some power of $h_1$ times $c_0$ times an element related to 
$\Sq^0$ of elements associated to the image of $J$.

It is curious that $c_0 \cdot i_1$ is divisible by $h_1^4$.
An obvious question for further study is to determine the
$h_1$-divisibility of $c_0$ times elements related to
$\Sq^0$ of elements associated to the image of $J$.
For example, what is the largest power of $h_1$ that divides
$g^2 i_1$?
\end{ex}

\begin{ex}
\label{ex:g^2}
Beware that $g^2$ and $g^3$ are not actually elements 
of the geometric 40-stem and 60-stem respectively.
Rather, it is only $\tau g^2$ and $\tau g^3$ that exist
(similarly, $g$ does not exist in the 20-stem, but $\tau g$ does exist).
The reason is that there are May differentials
taking $g^2$ to $h_1^8 h_5$, and $g^3$ to $h_1^6 i_1$.
In other words,
$\tau g^2$ and $\tau g^3$ are multiplicatively indecomposable elements.
More generally, we anticipate that the element $g^k$ does not exist
because it supports a May differential related to $\Sq^0$ of
an element in the image of $J$.
\end{ex}

\begin{ex}
There is an isomorphism from the cohomology of the classical Steenrod
algebra to the cohomology of the motivic Steenrod algebra over $\Spec \C$
concentrated in degrees of the form
$\left( 2s+f, f, \frac{s+f}{2} \right)$.
This isomorphism preserves all higher structure,
including squaring operations and Massey products.
See Section \ref{subsctn:Chow-deg-zero} for more details.

For example, the existence of the classical element $P h_2$
immediately implies that $h_3 g$ must be non-zero in the motivic setting;
no calculations are necessary.

Another example is that
$h_1^{2^k-1} h_{k+2}$ is non-zero motivically for all $k \geq 1$,
because $h_0^{2^k-1} h_{k+1}$ is non-zero classically.
\end{ex}


\begin{ex}
Many elements are $h_1$-local in the sense that they
support infinitely many multiplications by $h_1$.  
In fact, any product of the 
symbols $h_1$, $c_0$, $P$, $d_0$, $e_0$, and $g$, if it
exists, is non-zero.  This is detectable in the cohomology
of motivic $A(2)$ \cite{I1}.

Moreover, the element $B_1$ in the 46-stem is $h_1$-local, and any
product of $B_1$ with elements in the previous paragraph is again
$h_1$-local.  This is detectable in the $h_1$-local
cohomology of motivic $A(3)$, as described in
Theorem \ref{thm:A3-h1-local}.
We will see in \cite{I2} that these $h_1$-local computations
lead to interesting results about classical Adams differentials.

The $h_1$-localization of the 
cohomology of the motivic Steenrod algebra over $\Spec \C$
deserves further study.
It should be possible to make a full calculation 
in a much larger range than is possible for the
unlocalized cohomology of the motivic Steenrod algebra.  
Since localization is exact, one may localize $h_1$ in the $E_1$-term
of the May spectral sequence.  This makes the $E_1$-term much smaller,
although still complicated in sufficiently high degrees.
\end{ex}

\begin{ex}
The motivic analogue of ``wedge'' subalgebra \cite{MT}
appears to be more complicated than the classical version.
For example, none of the wedge elements support
multiplications by $h_0$ in the classical case.  
Motivically, many wedge elements do support $h_0$ multiplications.
The results in this paper naturally call for further study of 
the structure of the motivic wedge.  
\end{ex}

\subsection{Acknowledgements}

The author thanks R.\ R.\ Bruner for the use of his
extensive library of machine-assisted classical computations.
Many of the results in this article 
would have been impossible to discover without the guidance
of this data.
Similarly, the author thanks D.\ Dugger, whose machine-assisted
motivic computations give the $\M_2$-module structure
of the cohomology of the motivic Steenrod algebra.  

M.\ Tangora provided invaluable help with various technical 
computations in the classical May spectral sequence.

M.\ Mahowald and M.\ Behrens played especially pivotal roles
in giving the author the courage to attempt the calculation.

\section{Foundations}
\label{sctn:review}

Stable motivic homotopy is bigraded.  In a typical bidgree $(p,q)$,
we refer to $p$ as the geometric stem,
to $p-q$ as the topological stem, and to $q$ as the weight.

We will use the following notation extensively:
\begin{enumerate}
\item
$\M_2$ is the mod 2 motivic cohomology of $\Spec \C$.
\item
$A$ is the mod 2 motivic Steenrod algebra over $\Spec \C$.
\item
$\Ext$ is the trigraded ring $\Ext_A(\M_2,\M_2)$.
\item
$A_{\cl}$ is the classical mod 2 Steenrod algebra.
\item
$\Ext_{\cl}$ is the bigraded ring $\Ext_{A_{\cl}}(\F_2,\F_2)$.
\end{enumerate}

The following two deep theorems of Voevodsky are the starting points
of our calculations.

\begin{thm}[\cite{V1}]
$\M_2$ is the bigraded ring $\F_2[\tau]$, where
$\tau$ has bidegree $(0,1)$.
\end{thm}

\begin{thm}[\cite{V2} \cite{V3}]
The motivic Steenrod algebra $A$ is the $\M_2$-algebra generated by
elements $\Sq^{2k}$ and $\Sq^{2k-1}$ for all $k \geq 1$, of bidegrees
$(2k,k)$ and $(2k-1,k-1)$ respectively, and satisfying the following
relations for $a< 2b$:
\[ 
\Sq^a \Sq^b = 
\sum_{c} \binom{b-1-c}{a-2c} \tau^{?}\Sq^{a+b-c} \Sq^c.
\]
\end{thm}

The symbol $?$ stands for either $0$ or $1$, depending on which value
makes the formula balanced in weight.
See \cite{DI} for a more detailed discussion of the motivic Adem relations.


The $A$-module structure on $\M_2$ is trivial, i.e.,
every $\Sq^k$ acts by zero.  This follows for simple degree reasons.

\subsection{$\Ext$ groups}

We are interested in computing
$\Ext_A(\M_2,\M_2)$, which we abbreviate as $\Ext$.
This is a trigraded object.  We will consistently use degrees of the
form $(s,f,w)$, where:
\begin{enumerate}
\item
$f$ is the Adams filtration, i.e., the homological degree.
\item
$s+f$ is the internal geometric degree, i.e., corresponds to the
first coordinate in the bidegrees of $A$.
\item
$s$ is the geometric stem, i.e., the internal geometric degree minus
the Adams filtration.
\item
$w$ is the weight.
\end{enumerate}

Note that $\Ext^{*,0,*} = \Hom_A^{*,*}(\M_2,\M_2)$
is dual to $\M_2$.  We will abuse notation and write
$\M_2$ for this dual.  Beware that now $\tau$, which is really
the dual of the $\tau$ that we discussed earlier, has degree $(0,0,-1)$.
Since $\Ext$ is a module over $\Ext^{*,0,*}$, i.e., over $\M_2$,
we will always describe $\Ext$ as an $\M_2$-module.

The following result is the key tool for comparing classical
and motivic computations.  The point is that the motivic
and classical computations become the same after inverting $\tau$.

\begin{prop}[\cite{DI}]
\label{prop:compare-ext}
There is an isomorphism of rings
\[
\Ext \tens_{\M_2} \M_2 [\tau^{-1}] \iso
\Ext_{A_{\cl}} \tens_{\F_2} \F_2 [\tau,\tau^{-1}].
\]
\end{prop}

\subsection{The motivic May spectral sequence}

The classical May spectral sequence arises by filtering the
classical Steenrod algebra by powers of the augmentation ideal.
The same approach can be applied in the motivic setting to 
obtain the motivic May spectral sequence.  
Details appear in \cite{DI}.  Next we review the main points.

The motivic May spectral
sequence is quadruply graded.  We will always use gradings
of the form $(m,s,f,w)$, where $m$ is the May filtration,
and the other coordinates are as explained earlier.

Let $\Gr(A)$ be the associated graded algebra of $A$,
and let $\Gr(A_{\cl})$ be the associated graded algebra of $A_{\cl}$.

\begin{thm}
\label{thm:may}
The motivic May spectral sequence takes the form
\[
E_2 = 
\Ext^{(m,s,f,w)}_{\Gr(A)}(\M_2,\M_2) \Rightarrow
\Ext_A^{(s,f,w)}(\M_2,\M_2).
\]
\end{thm}

\begin{remark}
As in the classical May spectral sequence,
the odd differentials must be trivial for degree reasons.
\end{remark}

\begin{prop}
\label{prop:may-compare}
After inverting $\tau$,
there is an isomorphism of spectral sequences between
the motivic May spectral sequence of Theorem \ref{thm:may}
and the classical May spectral sequence, tensored over
$\F_2$ with $\F_2[\tau, \tau^{-1}]$.
\end{prop}

\begin{proof}
Start with the fact that $A [\tau^{-1}]$ is isomorphic to
$A_{\cl} \otimes_{\F_2} \F_2[\tau, \tau^{-1}]$, 
with the same May filtrations.
\end{proof}

This proposition means that differentials in the motivic May
spectral sequence must be compatible with the classical differentials.
This fact is critical to the success of our computations.

\subsection{$\Ext$ in Chow degree zero}
\label{subsctn:Chow-deg-zero}

\begin{defn}
Let $A'$ be the subquotient $\M_2$-algebra of $A$ generated by
$\Sq^{2k}$ for all $k \geq 0$, subject to the relation $\tau = 0$.
\end{defn}

\begin{lemma}
There is an isomorphism $A_{\cl} \map A'$ that takes
$\Sq^k$ to $\Sq^{2k}$.
\end{lemma}

The isomorphism takes elements of degree $n$ to elements
of bidegree $(2n,n)$.

\begin{proof}
Modulo $\tau$, the motivic Adem relation for $\Sq^{2a} \Sq^{2b}$
takes the form
\[ 
\Sq^{2a} \Sq^{2b} = 
\sum_{c} \binom{2b-1-2c}{2a-4c} \Sq^{2a+2b-2c} \Sq^2c.
\]
A standard fact from combinatorics says that
\[
\binom{2b-1-2c}{2a-4c} = \binom{b-1-c}{a-2c}
\]
modulo $2$.
\end{proof}

\begin{defn}
Let $M$ be a bigraded $A$-module.  The Chow degree
of an element $m$ in degree $(t,w)$ is equal to $t-2w$.
\end{defn}

The terminology arises from the fact that the Chow degree is
fundamental in Bloch's higher Chow group perspective on 
motivic cohomology \cite{Bl}.

\begin{defn}
Let $M$ be an $A$-module.  Define the $A'$-module
$\Ch_0(M)$ to be the subset of $M$ consisting of elements
of Chow degree zero, with $A'$-module structure induced from the
$A$-module structure on $M$.
\end{defn}

The $A'$-module structure on $\Ch_0(M)$ is well-defined since
$\Sq^{2k}$ preserves Chow degrees.

\begin{thm}
\label{thm:Chow-0}
There is an isomorphism from
$\Ext_{\cl}$ to the subalgebra of $\Ext$ consisting of elements
whose internal Chow degree is zero.
This isomorphism takes classical elements of degree
$(s,f)$ to motivic elements of degree $(2s+f,f,\frac{s+f}{2})$,
and it preserves all higher structure, including
products, squaring operations, and Massey products.
\end{thm}

\begin{remark}
The internal Chow degree of an element of $\Ext^{s,f,w}$ is
$s+f-2w$, since the internal geometric degree of such an 
element is $s+f$.
\end{remark}

\begin{proof}
There is a natural transformation
$\Hom_A( -, \M_2 ) \map \Hom_{A'} ( C_0(-), \F_2 )$,
since $\Ch_0(\M_2) = \F_2$.
Since $\Ch_0$ is an exact functor, the derived functor of the
right side is $\Ext_{A'} ( \Ch_0(-), \F_2 )$.
The universal property
of derived functors gives a natural transformation
$\Ext_A (-, \M_2 ) \map \Ext_{A'} ( \Ch_0(-), \F_2 )$.
Apply this natural transformation to $\M_2$ to obtain
$\Ext_A (\M_2, \M_2 ) \map \Ext_{A'} ( \Ch_0(\M_2), \F_2 )$.
The left side is $\Ext$, and the right side 
is isomorphic to $\Ext_{\cl}$ since $A'$ is isomorphic to $A_{\cl}$.

We have now obtained a map $\Ext \map \Ext_{\cl}$.
To verify that this map is an isomorphism on the Chow degree zero
part of $\Ext$, compare
the classical May spectral sequence with 
the part of the motivic May spectral sequence in Chow degree zero.
The motivic elements $h_{i0}$ have Chow degree $1$,
while the motivic elements $h_{ij}$ have Chow degree $0$
for $j > 0$.  It follows that the motivic $E_1$-term
in Chow degree zero is the polynomial algebra over $\F_2$
generated by $h_{ij}$ for $i > 0$ and $j > 0$.
This is isomorphic to the classical $E_1$-term, where
the motivic element $h_{ij}$ corresponds to the classical 
element $h_{i,j-1}$.
\end{proof}

\begin{remark}
Similar methods show that $\Ext$ is concentrated in positive 
Chow degree.  The map $\Ext \map \Ext_{\cl}$ constructed
in the proof takes elements in strictly positive Chow degree to zero.
Thus, the quotient of $\Ext$ by the strictly positive Chow
degree part is isomorphic to $\Ext_{\cl}$.
\end{remark}

\section{Computing with the motivic May spectral sequence}
\label{sctn:May}

\subsection{The $E_1$-term}

The $E_2$-term of the May spectral sequence is the cohomology
of an algebra.  In other words, the May spectral sequence really
starts with an $E_1$-term.
As described in \cite{DI},
the motivic $E_1$-term is essentially the same as the classical
$E_1$-term.  Specifically, the motivic $E_1$-term is 
a polyonomial algebra over $\M_2$ with
generators $h_{ij}$ for all $i > 0$ and $j \geq 0$,
where:
\begin{enumerate}
\item
$h_{i0}$ has degree $(i, 2^i-2, 1, 2^{i-1} - 1 )$.
\item
$h_{ij}$ has degree $(i, 2^j ( 2^i - 1) - 1, 1, 2^{j-1} ( 2^i - 1) )$
for $j > 0$.
\end{enumerate}
The $d_1$-differential is described by the formula:
\[
d_1 ( h_{ij} ) = \sum_{0 < k < i } h_{kj} h_{i-k,k+j}.
\]

\subsection{The $E_2$-term}

We now describe the $E_2$-term of the motivic
May spectral sequence.  As explained in \cite{DI},
it turns out that the motivic $E_2$-term is essentially
the same as the classical $E_2$-term.  The following
proposition makes this precise.

\begin{prop}[\cite{DI}]
\label{prop:assoc-graded}
There are graded ring isomorphisms
\begin{enumerate}[(a)]
\item 
$\Gr(A) \iso \Gr(A_{\cl})\tens_{\F_2} \F_2[\tau]$.
\item 
$\Ext_{\Gr(A)}(\M_2,\M_2) \iso 
\Ext_{\Gr(A_{cl})}(\F_2,\F_2)\tens_{\F_2} \F_2[\tau]$.
\end{enumerate}
\end{prop}

In other words, explicit generators and relations for
the $E_2$-term can be lifted directly from the
classical situation \cite{T}.  

Moreover,
because of Proposition \ref{prop:may-compare},
the values of $d_2$ can also be lifted from the classical
situation,
except that a few factors of $\tau$ show up to give the necessary weights.
For example, classically we have the differential
\[
d_2 ( b_{20} ) = h_1^3 + h_0^2 h_2.
\]
Motivically, this does not make sense, since
$b_{20}$ and $h_0^2 h_2$ have weight $2$, while $h_1^3$ has weight $3$.
Therefore, the motivic differential must be
\[
d_2 ( b_{20} ) = \tau h_1^3 + h_0^2 h_2.
\]

Table 1 lists the multiplicative
generators of the $E_2$-term
through the geometric 70-stem, and 
Table 2 lists a generating set of relations 
for the $E_2$-term in the same range.
Table 1 also gives the values of the May $d_2$ differential,
all of which are easily deduced from the classical situation
\cite{T}.

\subsection{The $E_4$-term}

Although the $E_2$-term is quite large, the $d_2$ differential
is also very destructive.  As a result, the $E_4$-page becomes
manageable.
We obtain the $E_4$-page by 
direct computation with the $d_2$ differential.

The multiplicative generators for the $E_4$-term 
through the geometric 70-stem break into two groups.
The first group consists of generators that survive to
$E_\infty$ and become multiplicative generators of $\Ext$;
these are listed in Table 14 as elements that occur on or before
the $E_4$-page.
The second group consists of generators that do not survive
to $E_\infty$;
these are listed in Table 3.
The reader should henceforth refer to these tables for notation.

\begin{remark}
As in \cite{T}, we use the notation
$B = b_{30} b_{31} + b_{21} b_{40}$.
\end{remark}

Having described the $E_4$-term, it is now necessary to
find the values of the $d_4$ differential on the multiplicative
generators.
Most of the values of $d_4$ follow
from comparison to the classical case \cite{T},
together with a few factors of $\tau$ to balance the weights.
There is only one differential that is not classical.

\begin{lemma}
\label{lem:d4-g}
$d_4 ( g ) = h_1^4 h_4$.
\end{lemma}

\begin{proof}
By the Chow degree zero isomorphism of
Theorem \ref{thm:Chow-0},
we know that $h_1^4 h_4$ cannot survive the
motivic May spectral sequence because $h_0^4 h_3$ 
is zero classically.  
There is only one possible differential
that can kill $h_1^4 h_4$.

See also \cite{DI} for a different proof of Lemma \ref{lem:d4-g}.
\end{proof}

Table 3 lists the values of the $d_4$ differential on multiplicative
generators of the $E_4$-page.

\subsection{The $E_6$-term}

We can now obtain the $E_6$-page
by direct computation with the $d_4$ differential and the Leibniz rule.

As for the $E_4$-term,
the multiplicative generators for the $E_6$-term 
through the geometric 70-stem break into two groups.
The first group consists of generators that survive to
$E_\infty$ and become multiplicative generators of $\Ext$;
these are listed in Table 14 as elements that occur on or before
the $E_6$-page.
The second group consists of generators that do not survive
to $E_\infty$;
these are listed in Table 4.

Having described the $E_6$-term, it is now necessary to
find the values of the $d_6$ differential on the multiplicative
generators.
Most of these values follow
from comparison to the classical case \cite{T},
together with a few factors of $\tau$ to balance the weights.
There are only a few differentials that are not classical.

\begin{lemma}
\mbox{}
\begin{enumerate}
\item
$d_6 ( x_{56} ) = h_1^2 h_5 c_0 d_0$.
\item
$d_6 ( P x_{56} ) = P h_1^2 h_5 c_0 d_0$.
\item
$d_6 ( B_{23} ) = h_1^2 h_5 d_0 e_0$.
\end{enumerate}
\end{lemma}

\begin{proof}
We have the relation $h_1 x_{56} = c_0 \phi$.
The $d_6$ differential on $\phi$ then implies that
$d_6 ( h_1 x_{56} ) = h_1^3 h_5 c_0 d_0$,
from which it follows that 
$d_6 ( x_{56} ) = h_1^2 h_5 c_0 d_0$.

The arguments for the other two differentials are similar,
using the relations
$h_1 \cdot P x_{56} = P c_0 \cdot \phi$
and
$h_1 B_{23} = e_0 \phi$.
\end{proof}

\begin{lemma}
\label{lem:d6-c0g^3}
$d_6 ( c_0 g^3 ) = h_1^{10} D_4$.
\end{lemma}

\begin{proof}
We start by computing that
$h_1^2 D_4$ belongs to $\langle c_0, h_4^2, h_3, h_1^3, h_1 h_3 \rangle$;
we will not need to worry about the indeterminacy.
One can use the May $d_2$ differential
to make this computation.  All of the threefold subbrackets are 
strictly zero, and one of the fourfold subbrackets is also strictly
zero.  However, $\langle c_0, h_4^2, h_3, h_1^3 \rangle$ equals
$\{ 0, h_1^2 h_5 e_0 \}$.

Now shuffle to obtain that $h_1^4 D_4$ is contained in
$c_0 \langle h_4^2, h_3, h_1^3, h_1 h_3, h_1^2 \rangle$.
The main point is that $h_1^4 D_4$ is divisible by $c_0$, which implies
that there is an extension $c_0 \cdot i_1 = h_1^4 D_4$
in $\Ext$ that is hidden in the May spectral sequence.

Since $h_1^6 i_1 = 0$, we conclude that $h_1^{10} D_4$ must
be zero in $\Ext$.
There is only one possible differential that can hit $h_1^{10} D_4$.
\end{proof}

\begin{remark}
The value of $d_6 ( \Delta h_0^2 Y )$ given in
\cite[Proposition 4.37(c)]{T} is incorrect because it is 
inconsistent with machine computations of $\Ext_{\cl}$ \cite{Br1}.  
The value for $d_6 (\Delta h_0^2 Y)$
given in our table is the only possibility that is consistent 
with the machine computations.
\end{remark}

Table 4 lists the values of the $d_6$ differential on multiplicative
generators of the $E_6$-page.


\subsection{The $E_8$-page}

We can now obtain the $E_8$-page
by direct computation with the $d_6$ differential and the Leibniz rule.

Once again,
the multiplicative generators for the $E_8$-term 
through the geometric 70-stem break into two groups.
The first group consists of generators that survive to
$E_\infty$ and become multiplicative generators of $\Ext$;
these are listed in Table 14 as elements that occur on or before
the $E_8$-page.
The second group consists of generators that do not survive
to $E_\infty$;
these are listed in Table 5.

Once we reach the $E_8$-term, we are nearly done.  There are 
just a few more higher differentials to deal with.

Having described the $E_8$-term, it is now necessary to
find the values of the $d_8$ differential on the multiplicative
generators.
Once again, most of these values follow
from comparison to the classical case \cite{T},
together with a few factors of $\tau$ to balance the weights.
There are only a few differentials that are not classical.

\begin{lemma}
\mbox{}
\begin{enumerate}
\item
$d_8 (g^2) = h_1^8 h_5$.
\item
$d_8 (w) = P h_1^5 h_5$.
\item
$d_8 ( \Delta c_0 g) = P h_1^4 h_5 c_0$.
\item
$d_8 ( Q_3 ) = h_1^4 h_5^2$.
\end{enumerate}
\end{lemma}

\begin{proof}
The Chow degree of $h_1^8 h_5$ is zero.  It follows from
Theorem \ref{thm:Chow-0} that
$h_1^8 h_5$ must be zero in $\Ext$, since
$h_0^8 h_4$ is zero classically.
There is only one differential that can possibly
hit $h_1^8 h_5$.

We now know that $P h_1^9 h_5 = 0$ in $\Ext$ since
$h_1^8 h_5 = 0$.  
There is only one differential that can hit this.
This shows that $d_8 ( w ) = P h_1^5 h_5$.

Using the relation $c_0 w = h_1 \cdot \Delta c_0 g$, it follows
that $d_8 ( h_1 \cdot \Delta c_0 g ) = P h_1^5 h_5 c_0$, and then that
$d_8 ( \Delta c_0 g ) = P h_1^4 h_5 c_0$.

The Chow degree of $h_1^4 h_5^2$ is zero.
Since $h_0^4 h_4^2$ is zero classically, it follows
from Theorem \ref{thm:Chow-0} that
$h_1^4 h_5^2$ must be zero in $\Ext$.
There is only one differential that can possibly hit
$h_1^4 h_5^2$.
\end{proof}

Table 5 lists the values of the $d_8$ differential on multiplicative
generators of the $E_8$-page.


\subsection{The $E_{12}$-page}

It turns out that the $d_{10}$ differential is zero through the 70-stem,
so $E_{10} = E_{12}$.  The next differential to consider
is $d_{12}$.

\begin{lemma}
The $d_{12}$-differential is zero on all multiplicative
generators of the $E_{12}$-term, except that
$d_{12} ( P^2 Q' ) = P h_0^{10} h_5 i$.
\end{lemma}

\begin{proof}
Compare to the classical case \cite{T}.
\end{proof}


\subsection{The $E_{16}$-page}

The $d_{14}$ differential is zero through the 70-stem,
so $E_{14} = E_{16}$.

\begin{lemma}
The $d_{16}$-differential is zero on all multiplicative
generators of the $E_{12}$-term, except that:
\begin{enumerate}
\item
$d_{16} ( P^4) = h_0^{16} h_5$.
\item
$d_{16} ( \Delta^2 h_4 ) = h_0^8 h_5^2$.
\end{enumerate}
\end{lemma}

\begin{proof}
Compare to the classical case \cite{T}.
\end{proof}


\subsection{The $E_{32}$-page}

Through the 70-stem, 
the last non-zero differential is $d_{32}$.

\begin{lemma}
The $d_{32}$-differential is zero on all multiplicative
generators of the $E_{32}$-term, except that
$d_{32} ( P^8) = h_0^{32} h_6$.
\end{lemma}

\begin{proof}
Compare to the classical case \cite{T}.
\end{proof}


\subsection{The $E_\infty$-term}
There are no more differentials to consider in our range,
and we have determined $E_\infty$.  

The multiplicative generators for the $E_\infty$-term 
through the geometric 70-stem break into two groups, but
in a different way than before.
The first group consists of generators that are still 
multiplicative generators in $\Ext$ 
after hidden extensions have been considered;
these are listed in Table 14.
The second group consists of multiplicative generators of
$E_\infty$ that become decomposable in $\Ext$ because
of a hidden extension;
these are listed in Table 9.

Before ending this section, we establish that the
multiplicative generator $P D_4$ of the $E_\infty$-term
becomes decomposable in $\Ext$ by a hidden extension.
This will be needed later in Lemma \ref{lem:h0-gr}(2)
to establish a hidden $h_0$ extension.

\begin{lemma}
\label{lem:Q2-c0}
$c_0 \cdot Q_2 = P D_4$.
\end{lemma}

\begin{proof}
First, compute that $h_1 Q_2 = \langle h_4, h_1^2 h_4, h_4, P h_1 \rangle$
with no indeterminacy;
use the May differentials $d_4 (\nu_1) = h_1^2 h_4^2$ and
$d_4 ( \Delta h_1 ) = P h_1 h_4$.  Beware that the subbracket
$\langle h_1^2 h_4, h_4, P h_1 \rangle$ is not strictly zero.

Next, compute that
$i_1 = \langle h_1^4, h_4, h_1^2 h_4, h_4 \rangle$ 
with no indeterminacy;
use the May differentials $d_4( g) = h_1^4 h_4$ and
$d_4 (\nu_1) = h_1^2 h_4^2$.
Beware that the subbracket
$\langle h_1^4, h_4, h_1^2 h_4\rangle$ is not strictly zero.

Now the shuffle
\[
h_1^4 \langle h_4, h_1^2 h_4, h_4, P h_1 \rangle =
\langle h_1^4, h_4, h_1^2 h_4, h_4 \rangle P h_1
\]
implies that $h_1^5 Q_2 = P h_1 \cdot i_1$.

Recall the hidden extension $c_0 \cdot i_1 = h_1^4 D_4$
from the proof of Lemma \ref{lem:d6-c0g^3}.
This implies that $h_1^5 c_0 Q_2 = P h_1^5 D_4$.
The desired relation now follows.
Note that $c_0 Q_2$ cannot equal
$\tau B_{23} + P D_4$ by comparison to the classical
case \cite{Br1}.
\end{proof}

\section{Hidden extensions}
\label{sctn:hidden-extn}

In order to pass from the $E_\infty$-term to $\Ext$,
we must resolve some hidden extensions.
In this section, we deal with all possible hidden extensions
by $\tau$, $h_0$, $h_1$, and $h_2$.
We will use several different tools, including:
\begin{enumerate}
\item
Classical hidden extensions \cite{Br1}.
\item
Shuffle relations with Massey products.
\item
Squaring operations in the sense of \cite{Ma2}.
\item
Theorem \ref{thm:Chow-0} for hidden extensions among
elements whose Chow degrees are zero.
\end{enumerate}

%


%

\subsection{Hidden $\tau$ extensions}
\label{subsctn:t-hidden}

By exhaustive search, the following results give all of the 
hidden $\tau$ extensions.

\begin{prop}
Table 10 lists all of the hidden
$\tau$ extensions through the geometric 70-stem.
\end{prop}

\begin{proof}
Many of the extensions follow by comparison to the classical case.
For example, there is a classical hidden extension
$h_0 \cdot e_0 g = h_0^4 x$.  This implies that
$\tau^2 \cdot h_0 e_0 g = h_0^4 x$ motivically.

Proofs for the more subtle cases are given below.
\end{proof}

\begin{lemma}
\label{lem:t^2-h2g^2}
\mbox{}
\begin{enumerate}
\item
$\tau \cdot \tau h_2 g^2 = P h_1^4 h_5$.
\item
$\tau \cdot \tau h_0 g^3 = P h_1^4 h_5 e_0$.
\end{enumerate}
\end{lemma}

\begin{proof}
Start with the relation $h_1 \cdot \tau g + h_2 f_0 = 0$,
and apply the squaring operation $\Sq^4$.
One needs that $\Sq^3 ( \tau g ) = P h_1^2 h_5$ \cite{Br2}.
The result is the first hidden extension.

For the second, multiply the first hidden extension by $e_0$.
\end{proof}

\begin{lemma}
\label{lem:t-B8}
\mbox{}
\begin{enumerate}
\item
$\tau \cdot B_8 = P h_5 d_0$.
\item
$\tau \cdot h_1^2 B_{21} = P h_5 c_0 d_0$.
\item
$\tau \cdot B_8 d_0 = h_0^4 X_3$.
\end{enumerate}
\end{lemma}

\begin{proof}
There is a classical hidden extension
$c_0 \cdot B_1 = P h_1 h_5 d_0$ \cite{Br1}.
Motivically, there is a non-hidden relation
$c_0 \cdot B_1 = h_1 B_8$.
It follows that
$\tau \cdot h_1 B_8 = P h_1 h_5 d_0$ motivically.

For the second hidden extension, multiply the
first hidden extension by $c_0$.
Note that $c_0 B_8 = h_1^2 B_{21}$ is
detected in the $E_\infty$-term of the May spectral sequence.

For the third hidden extension, 
multiply the first hidden extension by $d_0$,
and observe that 
$P h_5 d_0^2 = h_0^4 X_3$,
which is detected in the $E_\infty$-term of the May spectral 
sequence.
\end{proof}

\begin{lemma}
\mbox{}
\begin{enumerate}
\item
$\tau \cdot P u' = h_0^5 R_1$.
\item
$\tau \cdot P^2 u' = h_0^9 R$.
\item
$\tau \cdot P^3 u' = h_0^6 R_1'$.
\end{enumerate}
\end{lemma}

\begin{proof}
First, use the May $d_4$ differential to compute that 
$P u' = \langle u', h_0^3, h_0 h_3 \rangle$,
with no indeterminacy.  
Next, use the May $d_2$ differential to compute that
$\langle \tau, u', h_0^3 \rangle = Q'$,
again with no indeterminacy.

Use a shuffle to get that
$\tau \cdot P u' = h_0 h_3 Q'$.
Finally, there is a classical hidden extension 
$h_3 \cdot Q' = h_0^4 R_1$ \cite{Br1}, which implies that the same
formula holds motivically.

The argument for the second hidden extension is similar,
using the shuffle
\[
\tau \cdot P^2 u' = 
\tau \langle u', h_0^3, h_0^5 h_4 \rangle =
\langle \tau, u', h_0^3 \rangle h_0^5 h_4 =
h_0^5 h_4 Q'.
\]
The first equality comes from the May differential
$d_8 (P^2) = h_0^8 h_4$.
Also, we need the classical hidden
extension $h_4 \cdot Q' = h_0^4 R$ \cite{Br1}, which implies that the
same formula holds motivically.

The argument for the third hidden extension is also similar,
using the shuffle
\[
\tau \cdot P^3 u' =
\tau \langle u', h_0^3, h_0^3 i \rangle =
\langle \tau, u', h_0^3 \rangle h_0^3 i =
h_0^3 i Q'.
\]
The first equality comes from the May differential
$d_4 ( P^3 ) = h_0^6 i$.
Also, we need the classical hidden extension 
$i \cdot Q' = h_0^3 R_1'$ \cite{Br1}.
\end{proof}

\begin{lemma}
$\tau \cdot k_1 = h_2 h_5 n$.
\end{lemma}

\begin{proof}
First, use the May differential $d_4(\nu) = h_0^2 h_3^2$ to compute
$k = \langle d_0, h_3, h_0^2 h_3 \rangle$,
with no indeterminacy.  It follows from \cite{Mi} that
$\Sq^0 k = 
\langle \Sq^0 d_0, \Sq^0 h_3, \Sq^0 h_0^2 h_3 \rangle$,
with no indeterminacy.
In other words, 
$\Sq^0 k  = \langle \tau^2 d_1, h_4, \tau^2 h_1^2 h_4 \rangle$.
From the classical calculation \cite{Br2},
$\Sq^0 k$ also equals $\tau^3 h_2 h_5 n$.

On the other hand, the May differential 
$d_4 ( \nu_1 ) = h_1^2 h_4^2$ implies that 
$k_1 = \langle d_1, h_4, h_1^2 h_4 \rangle$,
with no indeterminacy.

This shows that $\tau^4 \cdot k_1 = \tau^3 h_2 h_5 n$
in $\Ext$, from which it follows that
$\tau \cdot k_1 = h_2 h_5 n$.
\end{proof}

%

\begin{remark}
In the 46-stem, $\tau \cdot u'$ does not equal
$\tau^2 d_0 l$.  This follows by comparison to 
$\Ext_{A(2)}$ \cite{I1}.
Similarly, in the 49-stem,
$\tau \cdot v'$ does not equal $\tau^2 e_0 l$.
\end{remark}

%

\subsection{Hidden $h_0$ extensions}
\label{subsctn:h0-hidden}

By exhaustive search, the following results give all of the
hidden $h_0$ extensions.

\begin{prop}
Table 11 lists all of the hidden
$h_0$ extensions through the geometric 70-stem.
\end{prop}

\begin{proof}
Many of the extensions follow by comparison to the classical case.
For example, there is a classical hidden extension
$h_0 \cdot r = s$.  This implies that
$h_0 \cdot r = s$ motivically as well.

Several other extensions are implied by the hidden
$\tau$ extensions established in 
Section \ref{subsctn:t-hidden}.
For example, the extensions $\tau \cdot P u' = h_0^4 S_1$
and $\tau \cdot \tau h_0 d_0^2 j = h_0^5 S_1$ imply
that $h_0 \cdot P u' = \tau h_0 d_0^2 j$.

Proofs for the more subtle cases are given below.
\end{proof}

\begin{lemma}
\mbox{}
\begin{enumerate}
\item
$h_0 \cdot u' = \tau h_0 d_0 l$.
\item
$h_0 \cdot v' = \tau h_0 e_0 l$.
\item
$h_0 \cdot P v' = \tau h_0 d_0^2 k$.
\item
$h_0 \cdot P^2 v' = \tau h_0 d_0^3 i$.
\end{enumerate}
\end{lemma}

\begin{proof}
These follow by comparison to $\Ext_{A(2)}$ \cite{I1}.
\end{proof}

\begin{lemma}
\label{lem:hidden-h0h2^2g}
\mbox{}
\begin{enumerate}
\item
$h_0 \cdot h_2^2 g = h_1^3 h_4 c_0$.
\item
$h_0 \cdot h_2^2 g^2 = h_1^7 h_5 c_0$.
\item
$h_0 \cdot h_2^2 g^3 = h_1^9 D_4$.
\end{enumerate}
\end{lemma}

\begin{proof}
For the first hidden extension, use the shuffle
\[
h_1^3 h_4 \anglb{ h_1, h_0, h_2^2 } = 
\anglb{ h_1^3 h_4, h_1, h_0 } h_2^2.
\]
Similarly, for the second hidden section, use the shuffle
\[
h_1^7 h_5 \anglb{ h_1, h_0, h_2^2 } = 
\anglb{ h_1^7 h_5, h_1, h_0 } h_2^2.
\]

For the third hidden extension, first
recall the hidden extension $c_0 \cdot i_1 = h_1^4 D_4$
from the proof of Lemma \ref{lem:d6-c0g^3}.
Use this relation to compute that
\[
h_1^9 D_4 = h_1^5 i_1 \langle h_1, h_2^2, h_0 \rangle =
\langle h_1^5 i_1, h_1, h_2^2 \rangle h_0.
\]
Finally, compute that $h_2^2 g^3 = \langle h_1^5 i_1, h_1, h_2^2 \rangle$
from the May differential $d_4 ( g^3) = h_1^6 i_1$.
\end{proof}

\begin{lemma}
\label{lem:h0-gr}
\mbox{}
\begin{enumerate}
\item
$h_0 \cdot g r = P h_1^3 h_5 c_0$.
\item
$h_0 \cdot l m = h_1^6 X_1$.
\item
$h_0 \cdot m^2 = h_1^5 c_0 Q_2$.
\end{enumerate}
\end{lemma}

\begin{remark}
The three parts may seem unrelated, but note that
$l m = e_0 g r$ and $m^2 = g^2 r$ on the $E_8$-page of the May spectral sequence.
\end{remark}

\begin{proof}
First compute that 
$e_0 r = \langle \tau^2 g^2, h_2^2, h_0 \rangle$
because of the May differential $d_4( \Delta h_2 g) = \tau^2 h_2^2 g^2$.
Next observe that
\[
h_2 \cdot e_0 r = \langle \tau^2 g^2, h_2^2, h_0 \rangle h_2 =
\langle \tau^2 g^2, h_2^2, h_0 h_2 \rangle =
\langle \tau^2 h_2 g^2, h_2, h_0 h_2 \rangle.
\]
None of these brackets have indeterminacy.

Use the relation $P h_1^4 h_5 = \tau^2 h_2 g^2$ from
Lemma \ref{lem:t^2-h2g^2}
to write
\[
h_2 \cdot e_0 r = \langle Ph_1^4 h_5, h_2, h_0 h_2 \rangle =
P h_1^3 h_5 \langle h_1, h_2, h_0 h_2 \rangle = 
P h_1^3 h_5 c_0.
\]
The last step is to show that $h_2 \cdot e_0 r = h_0 \cdot g r$.
This follows from the calculation
\[
h_0 \cdot g r = h_0 \langle h_1^3 h_4, h_1, r \rangle =
\langle h_0, h_1^3 h_4, h_1 \rangle r = h_2 e_0 \cdot r.
\]
This finishes the proof of part (1).

For part (2), we will prove below in Lemma \ref{lem:h1-th1G}
that $h_1^3 X_1 = P h_5 c_0 e_0$.
So we wish to show that $h_0 \cdot l m = P h_1^3 h_5 c_0 e_0$.
This follows immediately from part (1), using that
$l m = e_0 g r$.

The proof of part (3) is similar to the proof of part (1).
First, $l m = \langle \tau^2 g^3, h_2^2, h_0 \rangle$.
As above, this implies that
$h_2 l m = \langle \tau^2 h_2 g^3, h_2, h_0 h_2 \rangle$.
Now use the (not hidden) relation $\tau^2 h_2 g^3 = h_1^6 Q_2$
to deduce that $h_2 l m = h_1^5 c_0 Q_2$.
The desired formula now follows since $h_2 l = h_0 m$.
\end{proof}

\begin{lemma}
$h_0 \cdot h_0^2 B_{22} = P h_1 h_5 c_0 d_0$.
\end{lemma}

\begin{proof}
This follows from the hidden $\tau$ extension
$\tau \cdot h_1^3 B_{21} = P h_1 h_5 c_0 d_0$
that follows from Lemma \ref{lem:t-B8},
together with the relation $\tau h_1^3 = h_0^2 h_2$.
\end{proof}

\begin{remark}
On the $E_\infty$-page, there is a relation
$h_0^2 B_4 + \tau h_1 B_{21} = 0$.
However, in $\Ext$, the relation becomes
$h_0^2 B_4 + \tau h_1 B_{21} = g_2'$.
This follows from the analogous classical hidden relation \cite{Br1}.
Through the geometric 70-stem,
this is the only example of a hidden relation of the form
$h_0 \cdot x + h_1 \cdot y$,
$h_0 \cdot x + h_2 \cdot y$, or
$h_1 \cdot x + h_2 \cdot y$.
\end{remark}


\subsection{Hidden $h_1$ extensions}

By exhaustive search, the following results give all of the
hidden $h_1$ extensions.

\begin{prop}
Table 12 lists all of the hidden
$h_1$ extensions through the geometric 70-stem.
\end{prop}

\begin{proof}
Many of the extensions follow by comparison to the classical case.
For example, there is a classical hidden extension
$h_1 \cdot x = h_2^2 d_1$.
This implies that there is a motivic hidden
extension $h_1 \cdot x = \tau h_2^2 d_1$.

Proofs for the more subtle cases are given below.
\end{proof}

\begin{lemma}
\label{lem:h1-th1G}
\mbox{}
\begin{enumerate}
\item
$h_1 \cdot \tau h_1 G = h_5 c_0 e_0$.
\item
$h_1 \cdot h_1 B_3 = h_5 d_0 e_0$.
\item
$h_1 \cdot \tau P h_1 G = P h_5 c_0 e_0$.
\item
$h_1 \cdot h_1^2 X_3 = h_5 c_0 d_0 e_0$.
\end{enumerate}
\end{lemma}

\begin{proof}
The element $\tau h_1 G$ is represented in the May spectral sequence
by $\tau \Delta_1 h_1^3$, which equals
$\Delta_1 h_0^2 h_2$ on the May $E_4$-page.
We can then compute that
$\tau h_1 G = \langle h_5, h_2 g, h_0^2 \rangle$
using the May differential $d_4 ( \Delta_1 h_2 ) = h_2 h_5 g$.
Next shuffle to obtain
$h_1 \cdot \tau h_1 G =
\langle h_5, h_2 g, h_0^2 \rangle h_1 =
h_5 \langle h_2 g, h_0^2, h_1 \rangle$.
Finally, compute that $c_0 e_0 = \langle h_2 g, h_0^2, h_1 \rangle$
with the May differential $d_2 ( h_0(1)^2 b_{21} ) = h_2^3 h_0(1)^2$.
This establishes the first hidden extension.

For the second hidden extension,
begin by computing that
$h_1 h_5 d_0 e_0 = \langle h_5 c_0 e_0, h_0, h_2^2 \rangle$
using the May differential $d_2 (h_0(1)) = h_0 h_2^2$.
From part (1), this equals
$\langle \tau h_1^2 G, h_0, h_2^2 \rangle$, which equals
$h_1^2 \langle \tau G, h_0, h_2^2 \rangle$
because there is no indeterminacy.
This shows that $h_5 d_0 e_0$ is divisible by $h_1$.
The only possible hidden extension is
$h_1 \cdot h_1 B_3 = h_5 d_0 e_0$.

For the third hidden extension,
start with the relation $h_1 \cdot \tau P G = P h_1 \cdot \tau G$
because there is no possible hidden relation.
Therefore, using part (1),
\[
h_1^3 \cdot \tau P G = P h_1 \cdot h_1^2 \cdot \tau G 
= P h_1 h_5 c_0 e_0.
\]
It follows that $h_1 \cdot \tau P h_1 G = P h_5 c_0 e_0$.

For the fourth hidden extension,
use part (1) to conclude that
$h_5 c_0 d_0 e_0$ is divisible by $h_1$.
The only possibility is that
$h_1 \cdot h_1^2 X_3 = h_5 c_0 d_0 e_0$.
\end{proof}

\begin{lemma}
$h_1 \cdot h_1^2 B_6 = \tau h_2^2 d_1 g$.
\end{lemma}

\begin{proof}
Begin with the hidden extension
$h_1 \cdot x = \tau h_2^2 d_1$, which follows by comparison to the
classical case \cite{Br1}.
Next compute that
$\tau h_2^2 d_1 g = \langle \tau h_2^2 d_1, h_1^2, h_1^2 h_4 \rangle$,
using the May differential $d_4 (g) = h_1^4 h_4$.
It follows that
$\tau h_2^2 d_1 g = \langle h_1 x, h_1^2, h_1^2 h_4 \rangle$,
which equals
$h_1 \langle x, h_1^2, h_1^2 h_4 \rangle$
because there is no indeterminacy.
Therefore, $\tau h_2^2 d_1 g$ is divisible by $h_1$, and 
the only possibility is that
$h_1 \cdot h_1^2 B_6 = \tau h_2^2 d_1 g$.
\end{proof}

\begin{lemma}
$h_1 \cdot h_1 D_{11} = \tau^2 c_1 g^2$.
\end{lemma}

\begin{proof}
Begin by computing that
$h_1 D_{11} = \langle y, h_1^2, h_1^2 h_4 \rangle$,
using the May differential $d_4(g) = h_1^4 h_4$ and the relation
$\Delta h_3^2 g = \Delta h_1^2 d_1$.
Also recall the hidden extension
$h_1 \cdot y = \tau^2 c_1 g$, which follows by comparison to the
classical case \cite{Br1}.

It follows that
\[
h_1^2 D_{11} = \langle h_1 y, h_1^4, h_4 \rangle =
\langle \tau^2 c_1 g, h_1^4, h_4 \rangle
\]
because there is no indeterminacy.
Finally, use the May differential $d_4(g) = h_1^4 h_4$ again
to compute that
$\tau^2 c_1 g^2 = \langle \tau^2 c_1 g, h_1^2, h_1^2 h_4 \rangle$.
\end{proof}

\begin{lemma}
$h_1 \cdot C_0$ equals zero, not
$h_0 h_5 l$.
\end{lemma}

\begin{proof}
First compute that $C_0$ belongs to
$\langle h_0 h_3^2, h_0, h_1, \tau h_1 g_2 \rangle$
using the May differentials $d_4 ( \nu ) = h_0^2 h_3^2$ and 
$d_4 (x_{47} ) = \tau h_1^2 g_2$.
Note that the subbracket
$\langle h_0 h_3^2, h_0, h_1 \rangle$ is strictly zero,
but the subbracket
$\langle h_0, h_1, \tau h_1 g_2 \rangle$ equals
$\{ 0, h_0 h_2 g_2 \}$.
We will not need to compute the indeterminacy
in the fourfold bracket.

Now a shuffle implies that
$h_1 \cdot C_0$ belongs to
$\langle h_1, h_0 h_3^2, h_0, h_1 \rangle \tau h_1 g_2$.
For degree reasons, the bracket
$\langle h_1, h_0 h_3^2, h_0, h_1 \rangle$
consists of elements spanned by $f_0$ and $\tau h_1^3 h_4$.
But the products $h_1 \cdot f_0$ and $h_1 \cdot \tau h_1^3 h_4$
are both zero, so
$\langle h_1, h_0 h_3^2, h_0, h_1 \rangle \tau h_1 g_2$ must be zero.
Therefore, $h_1 \cdot C_0$ is zero.
\end{proof}

\begin{lemma}
$h_1 \cdot r_1 = s_1$.
\end{lemma}

\begin{proof}
This follows immediately from Theorem \ref{thm:Chow-0}
and the classical relation $h_0 \cdot r = s$.
\end{proof}

\begin{lemma}
$h_1 \cdot h_1^2 q_1 = h_0^4 X_3$.
\end{lemma}

\begin{proof}
Apply $\Sq^6$ to the relation
$h_2 r = h_1 q$ to obtain that
$h_3 r^2 = h_1^2 \Sq^5 (q)$.
Next, observe that $Sq^5 ( q ) = h_1 q_1$
by comparison to the classical case \cite{Br2}.

By comparison to the classical case, there is a relation
$h_3 r = h_0^2 x + \tau h_2^2 n$, so
$h_3 r^2 = h_0^2 r x$.
Finally, use the hidden extension $h_0 \cdot r = s$
and the non-hidden relation $s x = h_0^3 X_3$.
\end{proof}

\subsection{Hidden $h_2$ extensions}

By exhaustive search, the following results give all of the
hidden $h_2$ extensions.

\begin{prop}
Table 13 lists all of the hidden
$h_2$ extensions through the geometric 70-stem.
\end{prop}

\begin{proof}
Many of the extensions follow by comparison to the classical case.
For example, there is a classical hidden extension
$h_2 \cdot Q_2 = h_5 k$.  This implies that the same formula
holds motivically.

Also, many extensions are implied by hidden $h_0$ extensions
that we already established in Section \ref{subsctn:h0-hidden}.
For example, There is a hidden extension
$h_0 \cdot h_2^2 g = h_1^3 h_4 c_0$.
This implies that there is also a hidden extension
$h_2 \cdot h_0 h_2 g = h_1^3 h_4 c_0$.

Proofs for the more subtle cases are given below.
\end{proof}

\begin{remark}
We established the extensions
\begin{enumerate}
\item
$h_2 \cdot e_0 r = P h_1^3 h_5 c_0$
\item
$h_2 \cdot l m = h_1^5 c_0 Q_2$
\end{enumerate}
in the proof of Lemma \ref{lem:h0-gr}.
The extension
$h_2 \cdot k m = h_1^6 X_1$ follows from
Lemma \ref{lem:h0-gr} and the relation
$h_2 k = h_0 l$.
\end{remark}

\begin{lemma}
$h_2 \cdot h_2 B_2 = h_1 h_5 c_0 d_0$.
\end{lemma}

\begin{proof}
First compute that 
$h_2 B_2 = \langle g_2, h_0^3, h_2^2 \rangle$, with no indeterminacy,
using the May differential $d_6 ( Y ) = h_0^3 g_2$.
Then $h_2 \cdot h_2 B_2$ equals 
$\langle g_2, h_0^3, h_2^3 \rangle$, 
because there is no indeterminacy.
This bracket equals
$\langle g_2, h_0^3, h_1^2 h_3 \rangle$,
which equals
$\langle g_2, h_0^3, h_1 \rangle h_1 h_3$ 
since there is no indeterminacy.
Compute that the bracket $\langle g_2, h_0^3, h_1 \rangle$
equals $B_1$, using the May differential $d_6 ( Y) = h_0^3 g_2$.

We have now shown that $h_2 \cdot h_2 B_2$ equals
$h_1 h_3 \cdot B_1$.
It remains to show that there is a hidden extension
$h_3 \cdot B_1 = h_5 c_0 d_0$.
First observe that
$B_1 \cdot h_1^2 d_0 = h_1^3 B_{21}$ by a non-hidden relation.
This implies that 
$B_1 \cdot \tau h_1^2 d_0 = P h_1 h_5 c_0 d_0$ by 
Lemma \ref{lem:t-B8}.

Now there is a hidden extension
$h_3 \cdot P h_1 = \tau h_1^2 d_0$, so
$B_1 \cdot h_3 \cdot P h_1 = P h_1 h_5 c_0 d_0$.
The only possibility is that
$h_3 \cdot B_1 = h_5 c_0 d_0$.
\end{proof}

\begin{lemma}
$h_2 \cdot B_6 = \tau e_1 g$.
\end{lemma}

\begin{proof}
We will first prove that 
\[
\langle \tau, B_6, h_1^2 h_3 \rangle = h_2 C_0,
\]
with no indeterminacy.  The May differentials 
$d_2( b_{30} b_{40} h_1(1)) = \tau B_6$ and 
$d_2( B h_1 b_{21} h_1(1)) = h_1^2 h_3 B_6$
imply that the bracket equals
$\tau B h_1 b_{21} h_1(1) + h_1^2 h_3 b_{30} b_{40} h_1(1)$.
But this expression equals $h_2 C_0$ on the May $E_4$-page
because of the differential
\[
d_2 ( b_{21} b_{30} b_{40} h_1(1)) =
h_2^3 b_{30} b_{40} h_1(1) + \tau B h_1 b_{21} h_1(1) + 
h_1^3 b_{30} b_{22} b_{40}.
\]

This means that
\[
\langle \tau, B_6, h_2^3 \rangle = h_2 C_0.
\]
If $h_2 \cdot B_6$ were zero, then this would imply that
\[
\langle \tau, B_6, h_2 \rangle h_2^2 = h_2 C_0.
\]
However, $C_0$ cannot be divisible by $h_2$.
\end{proof}

\section{$h_1$-local cohomology of motivic $A(3)$}
\label{sctn:h1-local}

Because $h_1^k$ is non-zero in $\Ext$,
it makes sense to invert $h_1$ and compute $\Ext[h_1^{-1}]$.
In fact, $\Ext[h_1^{-1}]$ is practically 
computable in a much larger range than $\Ext$.
For our purposes, it is enough to consider
$\Ext_{A(3)} [h_1^{-1}]$ because this simpler calculation
captures everything that we need to know in the range under study.
Here $A(3)$ is the $\M_2$-subalgebra of the Steenrod algebra $A$
generated by $\Sq^1$, $\Sq^2$, $\Sq^4$, and $\Sq^8$.

Since localization is exact, we can invert $h_1$ in the
$E_2$-term of the motivic May spectral sequence for $A(3)$;
then we have a spectral sequence that converges to
$\Ext_{A(3)} [h_1^{-1}]$.

When we do this, we find that some of the usual generators can be
written in terms of other generators.  For example,
the relation $h_0(1)^2 = b_{20} b_{21} + h_1^2 b_{30}$
implies that $b_{30} = h_1^{-2} h_0(1)^2 + h_1^{-2} b_{20} b_{21}$.
Similarly, some relations simplify, and others become redundant.
For example, the relation $h_0 h_1 = 0$ implies that $h_0 = 0$.

\begin{prop}
The $h_1$-local $E_2$-term of the motivic May spectral
sequence for $A(3)$ is the $\M_2[h_1^{\pm 1}]$-algebra
generated by the elements in the following table,
subject to the single relation $h_3 h_0(1) = 0$.
The values of the May $d_2$ differential are also given in the table.
\end{prop}

\begin{center}
\begin{tabular}{|l|l|l|}
\hline
generator & $(m,s,f,w)$ & $d_2$ \\
\hline
$h_3$ & $(1,7,1,4)$ & \\
$b_{20}$ & $(4,4,2,2)$ & $\tau h_1^3$ \\
$h_0(1)$ & $(4,7,2,4)$ & \\
$b_{21}$ & $(4,10,2,6)$ & $h_1^2 h_3$ \\
$b_{31}$ & $(6,26,2,14)$ & \\
$b'_{40}$ & $(10,30,4,16)$ & \\
\hline
\end{tabular}
\end{center}

The element $b'_{40}$ is defined to be $h_1^2 b_{40} + b_{20} b_{31}$.

\begin{proof}
This follows immediately by comparison to the motivic May spectral
sequence for $A$.
\end{proof}

\begin{prop}
\label{prop:A3-Einfty}
The $h_1$-local $E_\infty$-term of the motivic May spectral
sequence for $A(3)$ is the $\frac{\M_2}{\tau}[h_1^{\pm 1}]$-algebra
generated by the elements in the following table,
subject to the single relation $h_1^2 e_0^2 + c_0^2 g = 0$.
\end{prop}

\begin{center}
\begin{tabular}{|l|l|l|}
\hline
generator & $(m,s,f,w)$ & $E_2$ description \\
\hline
$P$ & $(8,8,4,4)$ & $b_{20}^2$ \\
$c_0$ & $(5,8,3,5)$ & $h_1 h_0(1)$ \\
$e_0$ & $(8,17,4,10)$ & $b_{21} h_0(1)$ \\
$g$ & $(8,20,4,12)$ & $b_{21}^2$ \\
$b_{31}$ & $(6,26,2,14)$  & \\
$b'_{40}$ & $(10,30,4,16)$  & \\
\hline
\end{tabular}
\end{center}

\begin{proof}
After the $d_2$ differential, there are no more differentials,
and $E_3 = E_\infty$.
\end{proof}

At this point, we almost completely know $\Ext_{A(3)}[h_1^{-1}]$.
There is just one possible hidden extension.

\begin{lemma}
\label{lem:A3-hidden-extn}
$h_1^2 e_0^2 + c_0^2 g = h_1^4 b'_{40}$ 
in $\Ext_{A(3)}[h_1^{-1}]$,
\end{lemma}

\begin{remark}
The formula
has been verified by machine computation.
Unfortunately,
we have been unable to find a conceptual proof of 
Lemma \ref{lem:A3-hidden-extn}.  
Use of brackets and squaring operations has not yielded any
results.
\end{remark}

\begin{remark}
This obscure algebraic hidden extension has significant consequences
for classical Adams differentials, as shown in \cite{I2}.
\end{remark}

\begin{thm}
\label{thm:A3-h1-local}
$\Ext_{A(3)} [ h_1^{-1} ]$ is equal to 
the free polyonmial $\frac{\M_2}{\tau}[h_1^{\pm 1}]$-algebra
generated by the elements in the following table.
\end{thm}

\begin{center}
\begin{tabular}{|l|l|l|}
\hline
generator & $(s,f,w)$ & $E_2$ description \\
\hline
$P$ & $(8,4,4)$ & $b_{20}^2$ \\
$c_0$ & $(8,3,5)$ & $h_1 h_0(1)$ \\
$e_0$ & $(17,4,10)$ & $b_{21} h_0(1)$ \\
$g$ & $(20,4,12)$ & $b_{21}^2$ \\
$b_{31}$ & $(26,2,14)$  & \\
\hline
\end{tabular}
\end{center}

\begin{proof}
This follows immediately from Proposition \ref{prop:A3-Einfty}
and Lemma \ref{lem:A3-hidden-extn}.
\end{proof}

\begin{cor}
Through the geometric 70-stem, the $h_1$-local elements of
$\Ext$ are precisely the elements indicated in the charts
with red arrows of slope 1.
\end{cor}

\begin{proof}
Each of these elements is detected by the map
$\Ext [h_1^{-1} ] \map \Ext_{A(3)} [h_1^{-1}]$.
\end{proof}

We give one specific consequence of these $h_1$-local calculations.  
The following corollary is the algebraic key to a number of Adams differentials
beyond the 50-stem.

\begin{cor}
In $\Ext$, there is a hidden extension
$e_0^3 + d_0 \cdot e_0 g = h_1^5 B_1$.
\end{cor}

\begin{proof}
This extension can be detected by the map $\Ext \map \Ext_{A(3)}[h_1^{-1}]$.
Under the induced map on $E_\infty$-pages of the respective May spectral sequences,
we have that
$d_0$ goes to $h_1^{-2} c_0^2$;
$e_0$ goes to $e_0$; and
$e_0 g$ goes to $e_0 g$.

The map $\Ext \map \Ext_{A(3)}[h_1^{-1}]$ then takes
$d_0$ to $h_1^{-2} c_0^2$;
and $e_0$ to $e_0$.
However,
for degree reasons, it is possible that $e_0 g$ maps to
either $e_0 g$ or $e_0 g + h_1^3 c_0 b_{31}$.
(In fact, it maps to $e_0 g + h_1^3 c_0 b_{31}$, but we will not need
this.)

In either case,
$e_0^3 + d_0 \cdot e_0 g$ maps to
a non-zero element in $\Ext_{A(3)}[h_1^{-1}]$.
It follows that there is a hidden
extension, and there is only one possibility.
\end{proof}

\begin{landscape}

\section{Tables}
\label{sctn:table}

\begin{tabular}{llll}
\multicolumn{4}{c}{Table 1: May $E_2$-term generators in stems 0--70} \\
\hline
 & $(m,s,f,w)$ & $d_2$ & description \\
\hline
$h_0$ & $(1,0,1,0)$ & & $h_{10}$\\
$h_1$ & $(1,1,1,1)$ & & $h_{11}$ \\
$h_2$ & $(1,3,1,2)$ & & $h_{12}$ \\
$b_{20}$ & $(4,4,2,2)$ & $\tau h_1^3 + h_0^2 h_2$ & $h_{20}^2$ \\
$h_3$ & $(1,7,1,4)$ & & $h_{13}$ \\
$h_0(1)$ & $(4,7,2,4)$ & $h_0 h_2^2$ & $h_{20} h_{21} + h_{11} h_{30}$ \\
$b_{21}$ & $(4,10,2,6)$ & $h_2^3 + h_1^2 h_3$ & $h_{21}^2$ \\
$b_{30}$ & $(6,12,2,6)$ & $\tau h_1 b_{21} + h_3 b_{20}$ & $h_{30}^2$ \\
$h_4$ & $(1,15,1,8)$ & & $h_{14}$ \\
$h_1(1)$ & $(4,16,2,9)$ & $h_1 h_3^2$ & $h_{21} h_{22} + h_{12} h_{31}$ \\
$b_{22}$ & $(4,22,2,12)$ & $h_3^3 + h_2^2 h_4$ & $h_{22}^2$ \\
$b_{31}$ & $(6,26,2,14)$ & $h_4 b_{21} + h_2 b_{22}$ & $h_{31}^2$ \\
$b_{40}$ & $(8,28,2,14)$ & $h_4 b_{30} + \tau h_1 b_{31}$ & $h_{40}^2$ \\
$h_5$ & $(1,31,1,16)$ & & $h_{15}$ \\
$h_2(1)$ & $(4,34,2,18)$ & $h_2 h_4^2$ & $h_{22} h_{23} + h_{13} h_{32}$ \\
$h_0(1,3)$ & $(7,38,3,20)$ & $h_4^2 h_0(1) + h_0 h_2 h_2(1)$ &
  $h_{50} h_{11} h_{13} + h_{40} h_{11} h_{23} +$ \\
  & & & $+ h_{20} h_{41} h_{13} + h_{20} h_{31} h_{23}$ \\
$b_{23}$ & $(4,46,2,24)$ & $h_4^3 + h_3^2 h_5$ & $h_{23}^2$ \\
$h_0(1,2)$ & $(9,46,3,24)$ & $h_3 h_0(1,3)$ & 
   $h_{30} h_{31} h_{32} + h_{30} h_{41} h_{22} +$ \\
   & & & $+ h_{40} h_{21} h_{32} + h_{40} h_{41} h_{12} +$ \\
   & & & $+h_{50} h_{21} h_{22} + h_{50} h_{31} h_{12}$ \\
$b_{32}$ & $(6,54,2,28)$ & $h_5 b_{22} + h_3 b_{23}$ & $h_{32}^2$ \\
$b_{41}$ & $(8,58,2,30)$ & $h_5 b_{31} + h_2 b_{32}$ & $h_{41}^2$ \\
$b_{50}$ & $(10,60,2,30)$ & $h_5 b_{40} + \tau h_1 b_{41}$ & $h_{50}^2$ \\
$h_6$ & $(1,63,1,32)$ & & $h_{16}$ \\
$h_3(1)$ & $(4,70,2,36)$ & $h_3 h_5^2$ & $h_{23} h_{24} + h_{14} h_{33}$ \\   
\hline
\end{tabular}
\hfill
\begin{tabular}{ll}
\multicolumn{2}{c}{Table 2: May $E_2$-term relations in stems 0--46} \\
\hline
relation & $(m,s,f,w)$ \\
\hline
$h_0 h_1$ & $(2,1,2,1)$ \\
$h_1 h_2$ & $(2,4,2,3)$ \\
$h_2 b_{20} = h_0 h_0(1)$ & $(5,7,3,4)$ \\
$h_2 h_3$ & $(2,10,2,6)$ \\
$h_2 h_0(1) = h_0 b_{21}$ & $(5,10,3,6)$ \\
$h_3 h_0(1)$ & $(5,14,3,8)$ \\
$h_0(1)^2 = b_{20} b_{21} + h_1^2 b_{30}$ & $(8,14,4,8)$ \\
$h_0 h_1(1)$ & $(5,16,3,9)$ \\
$h_3 b_{21} = h_1 h_1(1)$ & $(5,17,3,10)$ \\
$b_{20} h_1(1) = h_1 h_3 b_{30}$ & $(8,20,4,11)$ \\
$h_3 h_4$ & $(2,22,2,12)$ \\
$h_3 h_1(1) = h_1 b_{22}$ & $(5,23,3,13)$ \\
$h_0(1) h_1(1)$ & $(8,23,4,13)$ \\
$b_{20} b_{22} = h_0^2 b_{31} + h_3^2 b_{30}$ & $(8,26,4,14)$ \\
$b_{22} h_0(1) = h_0 h_2 b_{31}$ & $(8,29,4,16)$  \\
$h_4 h_1(1)$ & $(5,31,3,17)$ \\
$h_1(1)^2 = b_{21} b_{22} + h_2^2 b_{31}$ & $(8,32,4,18)$ \\
$h_1 h_2(1)$ & $(5,35,3,19)$ \\
$h_4 b_{22} = h_2 h_2(1)$ & $(5,37,3,20)$ \\
$b_{20} h_2(1) = h_0 h_0(1,3)$ & $(8,38,4,20)$ \\
$h_2 h_0(1,3) = h_0 h_4 b_{31}$ & $(8,41,4,22)$ \\
$h_0(1) h_2(1) = h_0 h_4 b_{31}$ & $(8,41,4,22)$ \\
$b_{21} h_2(1) = h_2 h_4 b_{31}$ & $(8,44,4,24)$ \\
$h_0(1) h_0(1,3) = h_1^2 h_4 b_{40} + h_4 b_{20} b_{31}$ & $(11,45,5,24)$ \\
$h_4 h_5$ & $(2,46,2,24)$ \\
$b_{30} h_2(1) = h_0 h_0(1,2) + h_2 h_4 b_{40}$ & $(10,46,4,24)$ \\
\hline
\end{tabular}

\begin{tabular}{ll}
\multicolumn{2}{c}{Table 2: May $E_2$-term relations in stems 48--70} \\
\hline
relation & $(m,s,f,w)$ \\
\hline
$b_{21} h_0(1,3) = h_1^2 h_0(1,2) + h_4 b_{31} h_0(1)$ & $(11,48,5,26)$ \\
$h_4 h_2(1) = h_2 b_{23}$ & $(5,49,3,26)$ \\
$h_1(1) h_2(1)$ & $(8,50,4,27)$ \\
$b_{30} h_0(1,3) = b_{20} h_0(1,2) + h_4 b_{40} h_0(1)$ & $(13,50,5,26)$ \\
$b_{23} h_0(1) = h_4 h_0(1,3)$ & $(8,53,4,28)$ \\
$h_0(1) h_0(1,2) = h_4 b_{40} b_{21} + h_4 b_{30} b_{31}$ & 
    $(13,53,5,28)$ \\
    
$h_1(1) h_0(1,3) = h_1 h_3 h_0(1,2)$ & $(11,54,5,29)$ \\
$b_{21} b_{23} = h_1^2 b_{32} + h_4^2 b_{31}$ & $(8,56,4,30)$ \\
$b_{30} b_{23} = b_{20} b_{32} + h_0^2 b_{41} + h_4^2 b_{40}$ &
	$(10,58,4,30)$ \\
$b_{22} h_0(1,3) = h_3^2 h_0(1,2) + h_0 b_{31} h_2(1)$ & $(11,60,5,32)$ \\
$b_{32} h_0(1) = h_4 h_0(1,2) + h_0 h_2 b_{41}$ & $(10,61,4,32)$ \\
$b_{23} h_1(1) = h_1 h_3 b_{32}$ & $(8,62,4,33)$ \\
$h_5 h_2(1)$ & $(5,65,3,34)$ \\
$b_{22} b_{23} = h_2(1)^2 + h_3^2 b_{32}$ & $(8,68,4,36)$ \\
$h_5 h_0(1,3)$ & $(8,69,4,36)$ \\    
\hline
\end{tabular}
\hfill
\begin{tabular}{llll}
\multicolumn{4}{c}{Table 3: Temporary generators of the May $E_4$-term} \\
\hline
& $(m,s,f,w)$ & description & $d_4$ \\
\hline
$P$ & $(8, 8, 4, 4)$ & $b_{20}^2$ & $h_0^4 h_3$ \\
$\nu$ & $(7, 15, 3, 8)$ & $h_2 b_{30}$ & $h_0^2 h_3^2$ \\
$g$ & $(8, 20, 4, 12)$ & $b_{21}^2$ & $h_1^4 h_4$ \\
$\Delta$ & $(12, 24, 4, 12)$ & $b_{30}^2$ & $\tau^2 h_2 g + P h_4$ \\
$\nu_1$ & $(7, 33, 3, 18)$ & $h_3 b_{31}$ & $h_1^2 h_4^2$ \\
$x_{34}$ & $(7, 34, 5, 18)$ & $h_0^3 h_2(1) + h_0 h_4^2 b_{20}$ & \\
$x_{35}$ & $(10, 35, 4, 18)$ & $h_0 h_3 b_{40}$ & $x_{34}$ \\
$Y$ & $(16, 45, 6, 24)$ & $B h_0(1)$ & \\
$x_{47}$ & $(13, 47, 5, 25)$ & $h_2 b_{40} h_1(1)$ & $\tau h_1^2 g_2$ \\
$x_{49}$ & $(10, 49, 4, 26)$ & $h_2 h_0(1,2)$ & $h_0 h_3 c_2$ \\
$\phi$ & $(17, 49, 7, 27)$ & $B h_1 b_{21}$ & \\
$\Delta_1$ & $(12, 52, 4, 28)$ & $b_{31}^2$ & $h_5 g + h_3 g_2$ \\
$X$ & $(23,54,9,28)$ & $B h_0 b_{20} b_{30}$ & \\
$x_{56}$ & $(21,56,9,31)$ & $B h_1 b_{21} h_0(1)$ & \\
$\Gamma$ & $(16,56,4,28)$ & $b_{40}^2$ & 
	$\Delta h_5 + \tau^2 \Delta_1 h_2$ \\
$x_{56}'$ & $(24,56,10,30)$ & $P B b_{21}$ & \\
$\phi'$ & $(23,57,9,30)$ & $B h_0 b_{30} h_0(1)$ & \\
$x_{59}$ & $(19,59,7,31)$ & $h_2 b_{30} b_{40} h_1(1)$ & $\tau^2 e_1 g$ \\
$x_{63}$ & $(20,63,8,33)$ & $h_1 b_{20} b_{30} h_0(1,2) +$
	& $\tau h_5 d_0 e_0$ \\
	& & $+ \tau h_1^2 b_{31} b_{40} h_0(1) +$ & \\
	& & $+ \tau b_{20} b_{31}^2 h_0(1)$ & \\
$x_{65}$ & $(20,65,8,34)$ & $h_0 h_3 b_{20} b_{31} b_{40}$ & $h_0^3 A''$ \\
$x_{68}$ & $(16,68,6,36)$ & $h_3^2 b_{31} b_{40} + \tau b_{31}^2 h_1(1)$ &
	 $\tau s_1$ \\
$\nu_2$ & $(7,69,3,36)$ & $h_4 b_{32}$ & $h_2^2 h_5^2$ \\
$x_{69}$ & $(13,69,5,36)$ & $h_3 b_{40} h_2(1)$ & $h_0^2 d_2$ \\
\hline
\end{tabular}


\begin{tabular}{llll}
\multicolumn{4}{c}{Table 4: Temporary generators of the May $E_6$-term} \\
\hline
 & $(m,s,f,w)$ & description & $d_6$ \\
\hline
$P^2$ & $(16,16,8,8)$ & $$ & $$ \\
$\Delta h_3$ & $(13,31,5,16)$ & $$ & $$ \\
$P i$ & $(23,31,11,16)$ & $P^2 \nu$ & $h_0^5 s$ \\
$P r$ & $(22,38,10,20)$ & $P \Delta h_2^2$ & $h_0^6 x$ \\
$g^2$ & $(16,40,8,24)$ & $$ & $$ \\
$Y$ & $(16,45,6,24)$ & $$ & $h_0^3 g_2$ \\
$w$ & $(21,45,9,25)$ & $\Delta h_1 g$ & $$ \\
$\Delta^2$ & $(24,48,8,24)$ & $$ & $$ \\
$\phi$ & $(17,49,7,27)$ & $$ & $h_1^3 h_5 d_0$ \\
$\Delta c_0 g$ & $(25,52,11,29)$ &  &  \\
$X$ & $(23,54,9,28)$ & $$ & $P h_0 h_5 d_0$ \\
$P Q'$ & $(37,55,17,28)$ & $P^2 \Delta h_0^2 \nu$ & $h_0^9 X$ \\
$x_{56}$ & $(21,56,9,31)$ &  & $h_1^2 h_5 c_0 d_0$ \\
$x_{56}'$ & $(24,56,10,30)$ &  & $P h_1^2 h_5 d_0$ \\
$\phi'$ & $(23,57,9,30)$ &  & $P h_0 h_5 e_0$ \\
$P x_{56}$ & $(29,64,13,35)$ & $$ & $P h_1^2 h_5 c_0 d_0$ \\
$P x_{56}'$ & $(32,64,14,34)$ & $$ & $P^2 h_1^2 h_5 d_0$ \\
$B_{23}$ & $(24,65,10,36)$ & $Y g$ & $h_1^2 h_5 d_0 e_0$ \\
$P \phi'$ & $(31,65,13,34)$ & $$ & $P^2 h_0 h_5 e_0$ \\
$Q_3$ & $(13,67,5,36)$ & $\Delta_1 h_4$ & $$ \\
$c_0 g^3$ & $(29,68,15,41)$ & $$ & $h_1^{10} D_4$ \\
$\Delta h_0^2 Y$ & $(30,69,12,36)$ & $$ & 
	$\Delta h_0^5 g_2 + h_0 h_5 d_0 i$ \\
$\Gamma h_0 h_3^2$ & $(19,70,7,36)$ & $$ & $$ \\
\hline
\end{tabular}
\hfill
\begin{tabular}{lll}
\multicolumn{3}{c}{Table 5: Temporary generators for the May $E_8$-term} \\
\hline
element & $(m,s,f,w)$ & $d_8$ \\
\hline
$P^2$ & $(16,16,8,8)$ & $h_0^8 h_4$ \\
$\Delta h_3$ & $(13,31,5,16)$ & $h_0^4 h_4^2$ \\
$g^2$ & $(16,40,8,24)$ & $h_1^8 h_5$ \\
$w$ & $(21,45,9,25)$ & $P h_1^5 h_5$ \\
$\Delta^2$ & $(24,48,8,24)$ & $P^2 h_5$ \\
$\Delta c_0 g$ & $(25,52,11,29)$ & $P h_1^4 h_5 c_0$ \\
\hline
\end{tabular}


\begin{tabular}{llll}
\multicolumn{4}{c}
{Table 9: Temporary generators for the May $E_{\infty}$-term}\\
\hline
element & $(m,s,f,w)$ & description & occurs \\
\hline
$s$ & $(13, 30, 7, 16)$ & $P h_4 h_0(1) + h_0^3 b_{20} b_{31}$ & $4$ \\
$P^2 s$ & $(29, 46, 15, 24)$ & & $12$ \\
$S_1$ & $(25,54,11,28)$ & $h_0^2 X$ & $8$ \\
$g_2'$ & $(23,60,11,32)$ & $P \Delta_1 h_0^3$ & $6$ \\
$\tau P G$ & $(22,62,10,33)$ & $\tau P \Delta_1 h_1^2$ & $6$ \\
$P h_5 i$ & $(24,62,12,32)$ & $$ & $8$ \\
$P^4 s$ & $(45,62,23,32)$ & $$ & $32$ \\
$P D_4$ & $(22,65,10,35)$ & $$ & $6$ \\
$s_1$ & $(13,67,7,37)$ & $h_5 b_{21}^2 h_1(1) + h_1^3 b_{21} b_{32}$ 
	& $4$ \\
$P h_5^2$ & $(10,70,6,36)$ & & $6$ \\
$\tau P^2 G$ & $(30,70,14,37)$ & $$ & $12$ \\
$P^2 S_1$ & $(41,70,19,36)$ & $$ & $12$ \\
\hline
\multicolumn{4}{c}{}\\
\multicolumn{4}{c}{Table 10: Hidden $\tau$ extensions in stems 0--54} \\
\cline{1-3}
$(s,f,w)$ & $x$ & $\tau \cdot x$ \\
\cline{1-3}
$(30,11,16)$ & $P c_0 d_0$ & $h_0^4 s$ \\
$(37,9,20)$ & $\tau h_0 e_0 g$ & $h_0^4 x$ \\
$(37,10,20)$ & $\tau h_0^2 e_0 g$ & $h_0^5 x$ \\
$(41,5,22)$ & $h_1 f_1$ & $h_0^2 c_2$ \\
$(43,9,24)$ & $\tau h_2 g^2$ & $P h_1^4 h_5$ \\
$(46,19,24)$ & $P^3 c_0 d_0$ & $P^2 h_0^4 s$ \\
$(53,9,28)$ & $B_8$ & $P h_5 d_0$ \\
$(53,17,28)$ & $\tau P h_0 d_0^2 e_0$ & $h_0^7 x'$ \\
$(53,18,28)$ & $\tau P h_0^2 d_0^2 e_0$ & $h_0^8 x'$ \\
$(54,10,29)$ & $h_1 B_8$ & $P h_1 h_5 d_0$ \\
$(54,15,28)$ & $P u'$ & $h_0^4 S_1$ \\
$(54,16,28)$ & $\tau h_0 d_0^2 j$ & $h_0^5 S_1$ \\
$(54,17,28)$ & $\tau^2 P h_1 d_0^2 e_0$ & $h_0^6 S_1$ \\
\cline{1-3}
\end{tabular}
\hfill
\begin{tabular}{lll}
\multicolumn{3}{c}{Table 10: Hidden $\tau$ extensions in stems 58--70} \\
\hline
$(s,f,w)$ & $x$ & $\tau \cdot x$ \\
\hline
$(58,8,30)$ & $h_1 Q_2$ & $h_0^2 D_2$ \\
$(60,13,34)$ & $\tau h_0 g^3$ & $P h_1^4 h_5 e_0$ \\
$(61,12,33)$ & $h_1^2 B_{21}$ & $P h_5 c_0 d_0$ \\
$(61,13,32)$ & $x' c_0$ & $h_0^4 X_1$ \\
$(62,13,34)$ & $h_1^3 B_{21}$ & $P h_1 h_5 c_0 d_0$ \\
$(62,19,32)$ & $P^2 u'$ & $P h_0^7 h_5 i$ \\
$(62,20,32)$ & $\tau P h_0 d_0^2 j$ & $P h_0^8 h_5 i$ \\
$(62,21,32)$ & $\tau^2 P^2 h_1 d_0^2 e_0$ & $P h_0^9 h_5 i$ \\
$(62,27,32)$ & $P^5 c_0 d_0$ & $P^4 h_0^4 s$ \\
$(64,8,34)$ & $h_1 X_2$ & $h_0^2 A''$ \\
$(65,7,35)$ & $k_1$ & $h_2 h_5 n$ \\
$(65,9,34)$ & $h_1 h_3 Q_2$ & $h_0^2 h_3 D_2$ \\
$(66,12,34)$ & $h_1^2 q_1$ & $h_0^2 D_2'$ \\
$(67,13,36)$ & $B_8 d_0$ & $h_0^4 X_3$ \\
$(68,12,36)$ & $\tau h_0 h_2 B_{23}$ & $h_0^4 G_{21} + h_5 d_0 i$ \\
$(69,7,36)$ & $\tau h_1^2 Q_3$ & $h_0^3 p'$ \\
$(69,25,36)$ & $\tau P^3 h_0 d_0^2 e_0$ & $P^2 h_0^7 x'$ \\
$(69,26,36)$ & $\tau P^3 h_0^2 d_0^2 e_0$ & $P^2 h_0^8 x'$ \\
$(70,6,36)$ & $\tau h_2 Q_3$ & $P h_5^2$ \\
$(70,7,36)$ & $\tau h_0 h_2 Q_3$ & $P h_0 h_5^2$ \\
$(70,23,36)$ & $P^3 u'$ & $P^2 h_0^4 S_1$ \\
$(70,24,36)$ & $\tau P^2 h_0 d_0^2 j$ & $h_0^5 S_1$ \\
$(70,25,36)$ & $\tau P^2 h_0^2 d_0^2 j$ & $h_0^6 S_1$ \\
\hline
\end{tabular}


\begin{tabular}{lll}
\multicolumn{3}{c}{Table 11: Hidden $h_0$ extensions} \\
\hline
$(s,f,w)$ & $x$ & $h_0 \cdot x$ \\
\hline
$(26,7,16)$ & $h_2^2 g$ & $h_1^3 h_4 c_0$ \\
$(30,7,16)$ & $r$ & $s$ \\
$(46,11,28)$ & $h_2^2 g^2$ & $h_1^7 h_5 c_0$ \\
$(46,12,25)$ & $u'$ & $\tau h_0 d_0 l$ \\
$(46,15,24)$ & $i^2$ & $P^2 s$ \\
$(49,12,27)$ & $v'$ & $\tau h_0 e_0 l$ \\
$(50,11,28)$ & $g r$ & $P h_1^3 h_5 c_0$ \\
$(54,16,29)$ & $P u'$ & $\tau h_0 d_0^2 j$ \\
$(54,11,28)$ & $R_1$ & $S_1$ \\
$(56,11,29)$ & $Q_1$ & $\tau h_2 x'$ \\
$(57,16,31)$ & $P v '$ & $\tau h_0 d_0 e_0 j$ \\
$(60,8,32)$ & $B_3$ & $h_5 k$ \\
$(60,16,33)$ & $d_0 u'$ & $\tau h_0 e_0^2 j$ \\
$(62,20,33)$ & $P^2 u'$ & $\tau P h_0 d_0^2 j$ \\
$(62,12,32)$ & $h_0 R$ & $P h_5 i$ \\
$(62,13,34)$ & $h_0^2 B_{22}$ & $P h_1 h_5 c_0 d_0$ \\
$(62,23,32)$ & $P^2 i^2$ & $P^4 s$ \\
$(63,8,34)$ & $X_2$ & $h_5 l$ \\
$(63,15,38)$ & $h_0 h_2 g^3$ & $h_1^7 h_5 c_0 e_0$ \\
$(63,16,35)$ & $e_0 u'$ & $\tau h_0 e_0^2 k$ \\
$(64,8,34)$ & $h_2 A'$ & $\tau^2 d_1^2$ \\
$(64,15,34)$ & $U$ & $P h_2 x'$ \\
$(65,20,35)$ & $P^2 v'$ & $\tau h_0 d_0^3 i$ \\
$(66,15,40)$ & $h_2^2 g^3$ & $h_1^9 D_4$ \\
$(66,16,37)$ & $e_0 v'$ & $\tau h_0 e_0^2 l$ \\
$(67,14,38)$ & $l m$ & $h_1^6 X_1$ \\
$(68,15,36)$ & $h_0 G_{11}$ & $\tau h_1 d_0 x'$ \\
$(68,20,37)$ & $P d_0 u'$ & $\tau h_0 d_0^3 j$ \\
\hline
\end{tabular}
\hfill
\begin{tabular}{lll}
\multicolumn{3}{c}{Table 11: Hidden $h_0$ extensions} \\
\hline
$(s,f,w)$ & $x$ & $h_0 \cdot x$ \\
\hline
$(70,24,37)$ & $P^3 u'$ & $\tau P^2 h_0 d_0^2 j$ \\
$(70,15,40)$ & $m^2$ & $h_1^5 c_0 Q_2$ \\
$(70,19,36)$ & $h_0 R_1'$ & $P^2 S_1$ \\
\hline
\multicolumn{3}{c}{} \\
\multicolumn{3}{c}{Table 12: Hidden $h_1$ extensions} \\
\hline
$(s,f,w)$ & $x$ & $h_1 \cdot x$ \\
\hline
\end{tabular}
\hfill
\begin{tabular}{lll}
\multicolumn{3}{c}{Table 13: Hidden $h_2$ extensions} \\
\hline
$(s,f,w)$ & $x$ & $h_2 \cdot x$ \\
\hline
$(50,11,28)$ & $e_0 r$ & $P h_1^3 h_5 c_0$ \\
$(52,12,29)$ & $v'$ & $\tau h_0 e_0 m$ \\
$(54,9,30)$ & $h_2 B_2$ & $h_1 h_5 c_0 d_0$ \\
$(57,16,31)$ & $P u'$ & $\tau h_0 d_0 e_0 j$ \\
$(58,8,32)$ & $B_6$ & $\tau e_1 g$ \\
$(59,11,31)$ & $Q_1$ & $\tau h_0 B_{21}$ \\
$(60,8,32)$ & $Q_2$ & $h_5 k$ \\
$(60,9,32)$ & $h_0 Q_2$ & $h_0 h_5 k$ \\
$(60,16,33)$ & $P v'$ & $\tau h_0 e_0^2 j$ \\
$(62,13,34)$ & $h_0^2 B_{21}$ & $P h_1 h_5 c_0 d_0$ \\
$(63,8,34)$ & $B_3$ & $h_5 l$ \\
$(63,15,38)$ & $h_0^2 g^3$ & $h_1^7 h_5 c_0 e_0$ \\
$(63,16,35)$ & $d_0 u'$ & $\tau h_0 e_0^2 k$ \\
$(64,7,34)$ & $A+A'$ & $h_0 A''$ \\
$(64,8,34)$ & $h_0 (A+A')$ & $h_0^2 A''$ \\
$(64,8,35)$ & $B_7$ & $\tau d_1^2$ \\
$(64,14,36)$ & $k m$ & $h_1^6 X_1$ \\
$(65,20,35)$ & $P^2 u'$ & $\tau h_0 d_0^3 i$ \\
$(66,14,40)$ & $h_2 g^3$ & $h_1^9 D_4$ \\
$(66,16,37)$ & $e_0 u'$ & $\tau h_0 e_0^2 l$ \\
$(67,15,36)$ & $U$ & $h_0 d_0 x'$ \\
$(68,10,37)$ & $h_2 c_0$ & $\tau d_1 t$ \\
$(68,14,36)$ & $R_2$ & $\tau h_1 d_0 x'$ \\
$(68,20,37)$ & $P^2 v'$ & $\tau h_0 d_0^3 j$ \\
$(69,16,39)$ & $e_0 v'$ & $\tau h_0 e_0^2 m$ \\
$(70,15,40)$ & $l m$ & $h_1^5 c_0 Q_2$ \\
\hline
\end{tabular}

\end{landscape}


\begin{landscape}

\begin{tabular}{llll}
\multicolumn{4}{c}{Table 14: Generators for $\Ext$ in stems 0--30} \\
\hline
element & $(m,s,f,w)$ & description & occurs \\
\hline
$h_0$ & $(1, 0, 1, 0)$ &  & $1$ \\ 
$h_1$ & $(1, 1, 1, 1)$ &  & $1$ \\
$h_2$ & $(1, 3, 1, 2)$ &  & $1$ \\
$h_3$ & $(1, 7, 1, 4)$ &  & $1$ \\
$c_0$ & $(5, 8, 3, 5)$ & $h_1 h_0(1)$ & $4$ \\ 
$P h_1$ & $(9,9,5,5)$ & $$ & $6$ \\
$P h_2$ & $(9,11,5,6)$ & $$ & $6$ \\
$d_0$ & $(8, 14, 4, 8)$ & $h_0(1)^2$ & $4$ \\
$h_4$ & $(1, 15, 1, 8)$ & & $1$ \\
$P c_0$ & $(13,16,7,9)$ & $$ & $6$ \\
$e_0$ & $(8, 17, 4, 10)$ & $b_{21} h_0(1)$ & $4$ \\
$P^2 h_1$ & $(17, 17, 9, 9)$ & & $12$ \\ 
$f_0$ & $(8,18,4,10)$ & $h_2 \nu$ & $6$ \\
$c_1$ & $(5, 19, 3, 11)$ & $h_2 h_1(1)$ & $4$ \\
$P^2 h_2$ & $(17, 19, 9, 10)$ & & $12$ \\
$\tau g$ & $(8,20,4,11)$ & $\tau b_{21}^2$ & $6$ \\
$P d_0$ & $(16,22,8,12)$ & $$ & $6$ \\
$h_2 g$ & $(9,23,5,14)$ & $$ & $6$ \\
$i$ & $(15,23,7,12)$ & $P \nu$ & $6$ \\
$P^2 c_0$ & $(21,24,11,13)$ & $$ & $12$ \\
$P e_0$ & $(15,25,7,14)$ & $$ & $6$ \\
$P^3 h_1$ & $(25, 25, 13, 13)$ & & $12$ \\ 
$j$ & $(15, 26, 7, 14)$ & $h_0 b_{30} h_0(1)^2$ & $4$ \\
$h_3 g$ & $(9,27,5,16)$ & $$ & $6$ \\
$P^3 h_2$ & $(25, 27, 13, 14)$ & & $12$ \\
$k$ & $(15,29,7,16)$ & $d_0 \nu$ & $6$ \\
$r$ & $(14,30,6,16)$ & $\Delta h_2^2$ & $6$ \\
$P^2 d_0$ & $(24, 30, 12, 16)$ & & $12$ \\
\hline
\end{tabular}
\hfill
\begin{tabular}{llll}
\multicolumn{4}{c}{Table 14: Generators for $\Ext$ in stems 31--41} \\
\hline
element & $(m,s,f,w)$ & description & occurs \\
\hline
$h_5$ & $(1, 31, 1, 16)$ & & $1$ \\
$n$ & $(11, 31, 5, 17)$ & $h_2 b_{30} h_1(1)$ & $4$ \\
$d_1$ & $(8, 32, 4, 18)$ & $h_1(1)^2$ & $4$ \\
$q$ & $(14, 32, 6, 17)$ & $\Delta h_1 h_3$ & $12$ \\
$l$ & $(15,32,7,18)$ & $e_0 \nu$ & $6$ \\
$P^3 c_0$ & $(29, 32, 15, 17)$ & & $12$ \\
$p$ & $(8,33,4,18)$ & $h_0 \nu_1$ & $6$ \\
$P^2 e_0$ & $(24, 33, 12, 18)$ & & $12$ \\
$P^4 h_1$ & $(33,33,17,17)$ & $$ & $32$ \\
$P j$ & $(23,34,11,18)$ & $$ & $6$ \\
$m$ & $(15,35,7,20)$ & $g \nu$ & $6$ \\
$P^4 h_2$ & $(33,35,17,18)$ & $$ & $32$ \\
$t$ & $(12, 36, 6, 20)$ & $\tau b_{21}^2 h_1(1) + h_1^2 b_{22} b_{30}$ 
     & $4$ \\
$x$ & $(11, 37, 5, 20)$ & $h_2 b_{22} b_{30} + h_2^3 b_{40}$ & $4$ \\
$e_0 g$ & $(16,37,8,22)$ & $$ & $6$ \\
$e_1$ & $(8, 38, 4, 21)$ & $b_{22} h_1(1)$ & $4$ \\
$y$ & $(14,38,6,20)$ & $\Delta h_3^2$ & $8$ \\
$P^3 d_0$ & $(32, 38, 16, 20)$ & & $12$ \\
$c_1 g$ & $(13,39,7,23)$ & $$ & $6$ \\
$u$ & $(21,39,9,21)$ & $\Delta h_1 d_0$ & $6$ \\
$P^2 i$ & $(31, 39, 15, 20)$ & & $12$ \\
$f_1$ & $(8,40,4,22)$ & $h_3 \nu_1$ & $6$ \\
$\tau g^2$ & $(16, 40, 8, 23)$ & & $12$ \\
$P^4 c_0$ & $(37,40,19,21)$ & $$ & $32$ \\
$c_2$ & $(5, 41, 3, 22)$ & $h_3 h_2(1)$ & $4$ \\
$z$ & $(22,41,10,22)$ & $\Delta h_0^2 e_0$ & $6$ \\
$P^3 e_0$ & $(32, 41, 16, 22)$ & & $12$ \\
$P^5 h_1$ & $(41,41,21,21)$ & $$ & $32$ \\
\hline
\end{tabular}

\begin{tabular}{llll}
\multicolumn{4}{c}{Table 14: Generators for $\Ext$ in stems 42--53} \\
\hline
element & $(m,s,f,w)$ & description & occurs \\
\hline
$v$ & $(21,42,9,23)$ & $\Delta h_1 e_0$ & $6$ \\
$P^2 j$ & $(31, 42, 15, 22)$ & & $12$ \\
$h_2 g^2$ & $(17, 43, 9, 26)$ & & $12$ \\
$P^5 h_2$ & $(41,43,21,22)$ & $$ & $32$ \\
$g_2$ & $(8, 44, 4, 24)$ & $b_{22}^2$ & $4$ \\
$\tau w$ & $(21, 45, 9, 24)$ & $\tau \Delta h_1 g$ & $12$ \\
$N$ & $(18,46,8,25)$ & $\Delta h_2 c_1$ & $6$ \\
$B_1$ & $(17,46,7,25)$ & $Y h_1$ & $8$ \\
$h_3 g^2$ & $(17, 47, 9, 28)$ & & $12$ \\
$u'$ & $(25,46,11,25)$ & $\Delta c_0 d_0$ & $6$ \\
$P^4 d_0$ & $(40,46,20,24)$ & $$ & $32$ \\
$Q'$ & $(29,47,13,24)$ & $P \Delta h_0^2 \nu$ & $6$ \\
$P u$ & $(29,47,13,25)$ & $P \Delta h_1 d_0$ & $6$ \\
$B_2$ & $(17,48,7,26)$ & $Y h_2$ & $8$ \\
$P^5 c_0$ & $(45,48,23,25)$ & $$ & $32$ \\
$v'$ & $(25,49,11,27)$ & $\Delta c_0 e_0$ & $6$ \\
$P^4 e_0$ & $(40,49,20,26)$ & $$ & $32$ \\
$P^6 h_1$ & $(49,49,25,25)$ & $$ & $32$ \\
$C$ & $(14,50,6,27)$ & $h_2 x_{47}$ & $6$ \\
$g r$ & $(22,50,10,28)$ & $\Delta h_2^2 g$ & $6$ \\
$P v$ & $(29,50,13,27)$ & $P \Delta h_1 e_0$ & $6$ \\
$P^3 j$ & $(39, 50, 19, 26)$ & & $6$ \\
$G_3$ & $(21,51,9,28)$ & $\Delta h_3 g$ & $6$ \\
$g n$ & $(19,51,9,29)$ & $$ & $6$ \\
$P^6 h_2$ & $(49,51,25,26)$ & $$ & $32$ \\
$D_1$ & $(11,52,5,28)$ & $h_2 x_{49}$ & $6$ \\
$d_1 g$ & $(16,52,8,30)$ & $$ & $6$ \\
$i_1$ & $(15,53,7,30)$ & $g \nu_1$ & $6$ \\
\hline
\end{tabular}
\hfill
\begin{tabular}{llll}
\multicolumn{4}{c}{Table 14: Generators for $\Ext$ in stems 53--60} \\
\hline
element & $(m,s,f,w)$ & description & occurs \\
\hline
$B_8$ & $(21,53,9,29)$ & $Y c_0$ & $8$ \\
$x'$ & $(24,53,10,28)$ & $P Y$ & $6$ \\
$\tau G$ & $(14,54,6,29)$ & $\tau \Delta_1 h_1^2$ & $6$ \\
$R_1$ & $(26,54,10,28)$ & $\Delta^2 h_2^2$ & $12$ \\
$P u'$ & $(33,54,15,29)$ & $P \Delta c_0 d_0$ & $6$\\
$P^5 d_0$ & $(48,54,24,28)$ & $$ & $32$ \\
$B_6$ & $(17,55,7,30)$ & $B h_1 h_1(1)$ & $4$ \\
$g m$ & $(23,55,11,32)$ & $g^2 \nu$ & $6$ \\ 
$P^2 u$ & $(37,55,17,29)$ & $$ & $12$ \\
$P^4 i$ & $(47,55,23,28)$ & $$ & $32$ \\
$g t$ & $(20,56,10,32)$ &  &  $6$ \\
$Q_1$ & $(26,56,10,29)$ & $\Delta^2 h_1 h_3$ & $12$ \\
$P^6 c_0$ & $(53,56,27,29)$ & $$ & $32$ \\
$D_4$ & $(14,57,6,31)$ & $h_1 b_{21} h_0(1,2)$ & $4$ \\
$Q_2$ & $(19,57,7,30)$ & $\Delta \nu_1$ & $6$ \\
$D_{11}$ & $(21,57,9,31)$ & $\Delta h_1 d_1$ & $6$  \\
$P v'$ & $(33,57,15,31)$ & $P \Delta c_0 e_0$ & $6$  \\
$P^5 e_0$ & $(48,57,24,30)$ & $$ & $32$ \\
$P^7 h_1$ & $(57,57,29,29)$ & $$ & $32$ \\
$D_2$ & $(16,58,6,30)$ & $h_0 b_{30} h_0(1,2)$ & $4$ \\
$e_1 g$ & $(16,58,8,33)$ &  & $6$  \\
$P^2 v$ & $(37,58,17,31)$ & $$ & $12$ \\
$P^4 j$ & $(47,58,23,30)$ & $$ & $32$ \\
$j_1$ & $(15,59,7,33)$ & $h_1 b_{21} b_{22} b_{31}$ & $4$ \\
$B_{21}$ & $(24,59,10,32)$ & $Y d_0$ & $8$ \\
$c_1 g^2$ & $(21,59,11,35)$ & $$ & $12$ \\
$P^7 h_2$ & $(57,59,29,30)$ & $$ & $32$ \\
$B_3$ & $(17,60,7,32)$ & $Y h_4$ & $8$ \\
\hline
\end{tabular}

\begin{tabular}{llll}
\multicolumn{4}{c}{Table 14: Generators for $\Ext$ in stems 60--65} \\
\hline
element & $(m,s,f,w)$ & description & occurs \\
\hline
$B_4$ & $(23,60,9,32)$ & $Y \nu$ & $6$ \\
$\tau g^3$ & $(24,60,12,35)$ & $$ & $12$ \\
$h_0 g^3$ & $(25,60,13,36)$ & $$ & $12$ \\
$D_3$ & $(10,61,4,32)$ & $h_4 h_0(1,2)$ & $4$ \\
$A$ & $(16,61,6,32)$ & $h_2 b_{30} h_0(1,2)$ & $4$ \\
$A'$ & $(16,61,6,32)$ & $h_2 b_{30} h_0(1,2) + h_0 h_3 b_{31} b_{40}$ 
	& $4$ \\
$B_7$ & $(17,61,7,33)$ & $B h_1 b_{22}$ & $4$ \\
$X_1$ & $(23,61,9,32)$ & $\Delta x$ & $6$  \\
$H_1$ & $(13,62,5,33)$ & $h_1(1) h_0(1,2)$ & $4$ \\
$C_0$ & $(20,62,8,33)$ & $h_2 x_{59}$ & $6$ \\
$E_1$ & $(20,62,8,33)$ & $\Delta e_1$ & $6$ \\
$B_{22}$ & $(24,62,10,34)$ & $Y e_0$ & $8$ \\
$R$ & $(26,62,10,32)$ & $\Delta^2 h_3^2$ & $12$ \\
$P^2 u'$ & $(41,62,19,33)$ & $$ & $12$ \\
$P^6 d_0$ & $(56,62,28,32)$ & $$ & $32$ \\
$h_6$ & $(1,63,1,32)$ & & $1$ \\
$C'$ & $(17,63,7,34)$ & $h_2 b_{40} h_1(1)^2$ & $4$ \\
$X_2$ & $(17,63,7,34)$ & 
	$\tau h_1 b_{21} b_{31}^2 + h_1^2 h_3 b_{31} b_{40}$ & $4$ \\
$h_2 g^3$ & $(25,63,13,38)$ & $$ & $12$ \\
$P^3 u$ & $(45,63,21,33)$ & $$ & $12$ \\
$A''$ & $(14,64,6,34)$ & $h_0 b_{31} h_0(1,3)$ & $4$ \\
$q_1$ & $(26,64,10,33)$ & $\Delta^2 h_1 h_4$ & $32$ \\
$U$ & $(34,64,14,34)$ & $\Delta^2 h_1^2 d_0$ & $12$ \\
$P^7 c_0$ & $(61,64,31,33)$ & $$ & $32$ \\
$k_1$ & $(15,65,7,36)$ & $d_1 \nu_1$ & $6$ \\
$\tau B_{23}$ & $(24,65,10,35)$ & $$ & $8$ \\
$R_2$ & $(33,65,13,34)$ & $\Delta^2 h_0 e_0$ & $12$ \\
\hline
\end{tabular}
\hfill
\begin{tabular}{llll}
\multicolumn{4}{c}{Table 14: Generators for $\Ext$ in stems 65--69} \\
\hline
element & $(m,s,f,w)$ & description & occurs \\
\hline
$\tau g w$ & $(29,65,13,36)$ & $$ & $12$ \\
$P^2 v'$ & $(41,65,19,35)$ & $$ & $12$ \\
$P^6 e_0$ & $(56,65,28,34)$ & $$ & $32$ \\
$P^8 h_1$ & $(65,65,33,33)$ & $$ & $34$ \\
$r_1$ & $(14,66,6,36)$ & $\Delta_1 h_3^2$ & $6$ \\
$\tau G_0$ & $(17,66,7,35)$ & 
	$\tau b_{21}^2 h_0(1,2) +$ & $4$ \\
	& & $+h_1 h_3 b_{30} h_0(1,2)$ & \\
$\tau B_5$ & $(24,66,10,35)$ & $\tau b_{21}^2 b_{30}^2 b_{22} +$ & $4$ \\
	& & $+ \tau h_2^2 b_{21}^2 b_{30} b_{40} + h_1 h_3 b_{30}^3 b_{22}$ & \\
$D_2'$ & $(24,66,10,34)$ & $P D_2 + \Delta h_0 h_3 x_{35}$ & $6$ \\
$P^3 v$ & $(45,66,21,35)$ & $$ & $12$ \\
$P^5 j$ & $(55,66,27,34)$ & $$ & $32$ \\
$n_1$ & $(11,67,5,36)$ & $h_3 b_{31} h_2(1)$ & $4$ \\
$\tau Q_3$ & $(13,67,5,35)$ & $$ & $12$ \\
$h_0 Q_3$ & $(14,67,6,36)$ & $$ & $12$ \\
$C''$ & $(21,67,9,37)$ & $g x_{47}$ & $6$ \\
$X_3$ & $(21,67,9,36)$ & $\Delta_1 h_0^2 \nu + \tau \Delta_1 h_1 d_0$ 
	& $6$ \\
$C_{11}$ & $(29,67,11,35)$ & $\Delta^2 c_1$ & $12$ \\
$h_3 g^3$ & $(25,67,13,40)$ & $$ & $12$ \\
$P^8 h_2$ & $(65,67,33,34)$ & $$ & $34$ \\
$d_2$ & $(8,68,4,36)$ & $h_2(1)^2$ & $4$ \\
$G_{21}$ & $(20,68,8,36)$ & $\Delta g_2$ & $6$ \\
$h_2 B_{23}$ & $(25,68,11,38)$ & $$ & $8$ \\
$G_{11}$ & $(33,68,13,36)$ & $\Delta^2 h_2 e_0$ & $12$ \\
$p'$ & $(8,69,4,36)$ & $h_0 \nu_2$ & $6$ \\
$D_3'$ & $(18,69,8,37)$ & 
	$h_1 b_{20} b_{31} h_0(1,3) +$ & $4$ \\
	& & $+h_1^3 b_{40} h_0(1,3)$ & \\
$h_2 G_0$ & $(18,69,8,38)$ & $$ & $6$ \\
\hline
\end{tabular}

\begin{tabular}{llll}
\multicolumn{4}{c}{Table 14: Generators for $\Ext$ in stems 69--70} \\
\hline
element & $(m,s,f,w)$ & description & occurs \\
\hline
$P(A+A')$ & $(24,69,10,36)$ & $$ & $6$ \\
$h_2 B_5$ & $(25,69,11,38)$ & $$ & $8$ \\
$\tau W_1$ & $(33,69,13,36)$ & $\tau \Delta^2 h_1 g$ & $12$ \\
$P^2 x'$ & $(40,69,18,36)$ & $$ & $12$ \\
$p_1$ & $(8,70,4,37)$ & $h_1 \nu_2$ & $6$ \\
$h_2 Q_3$ & $(14,70,6,38)$ & $$ & $12$ \\
$R_1'$ & $(41,70,17,36)$ & $\Delta^2 P h_0 d_0$ & $12$ \\
$P^3 u'$ & $(49,70,23,37)$ & $$ & $12$ \\
$P^7 d_0$ & $(64,70,32,36)$ & $$ & $32$ \\
\hline
\end{tabular}

\end{landscape}


\newpage

\section{$\Ext$ charts}
\label{sctn:Ext}

The charts show the cohomology of the motivic Steenrod algebra
over $\Spec \C$.

Here is a key for reading the charts.
\begin{enumerate}
\item
An element of degree $(s,f,w)$ appears at coordinates
$(s,f)$.
\item
Black dots indicate copies of $\M_2$.
\item
Red dots indicate copies of $\M_2 / \tau$.
\item
Blue dots indicate copies of $\M_2 / \tau^2$.
\item
Green dots indicate copies of $\M_2 / \tau^3$.
\item
Vertical lines indicate $h_0$ multiplications.  These lines might be
black, red, blue, or green, depending on the $\tau$ torsion of the target.
\item
Lines of slope $1$ indicate $h_1$ multiplications.  
These lines might
be black, red, blue, or green, depending on the $\tau$ torsion of the 
target.
\item
Lines of slope $1/3$ indicate $h_2$ multiplications.
These lines might
be black, red, blue, or green, depending on the $\tau$ torsion of the 
target.
\item
Red arrows indicate an infinite tower of $h_1$ multiplications,
all of which are annihilated by $\tau$.
\item
Magenta lines indicate that an extension hits $\tau$ times a generator.
For example, $h_0 \cdot h_0 h_2 = \tau h_1^3$ in the geometric 3-stem.
\item
Orange lines indicate that an extension hits $\tau^2$ times a generator.
For example, $h_0 \cdot h_0^3 x = \tau^2 h_0 e_0 g$ in the 
geometric 37-stem.
\item
Dotted lines indicate that the extension is hidden in the May
spectral sequence.
\item
Squares indicate that there is a $\tau$ extension that
is hidden in the May spectral sequence.
For example, $\tau \cdot P c_0 d_0 = h_0^5 r$
in the geometric 30-stem.
\end{enumerate}

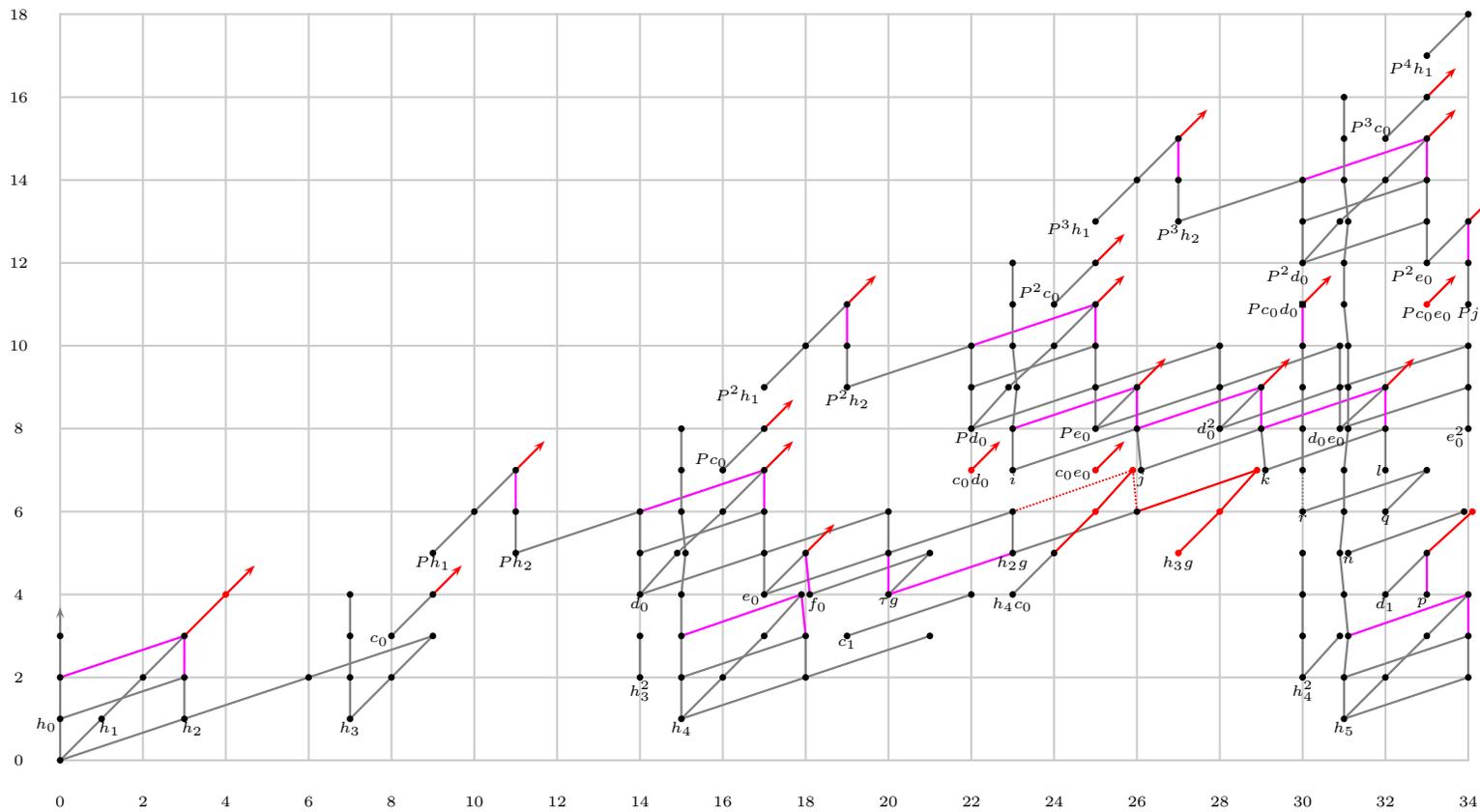
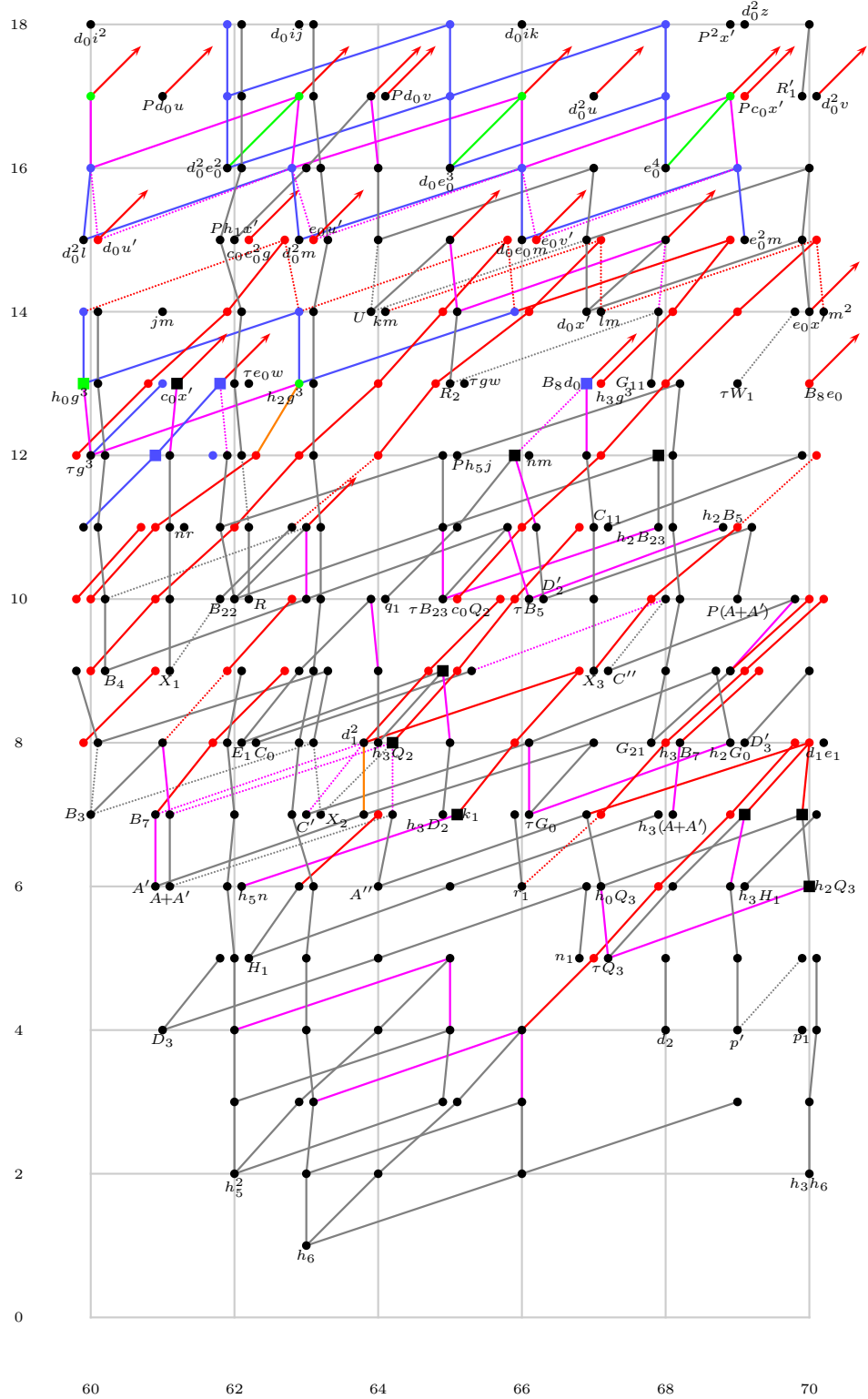
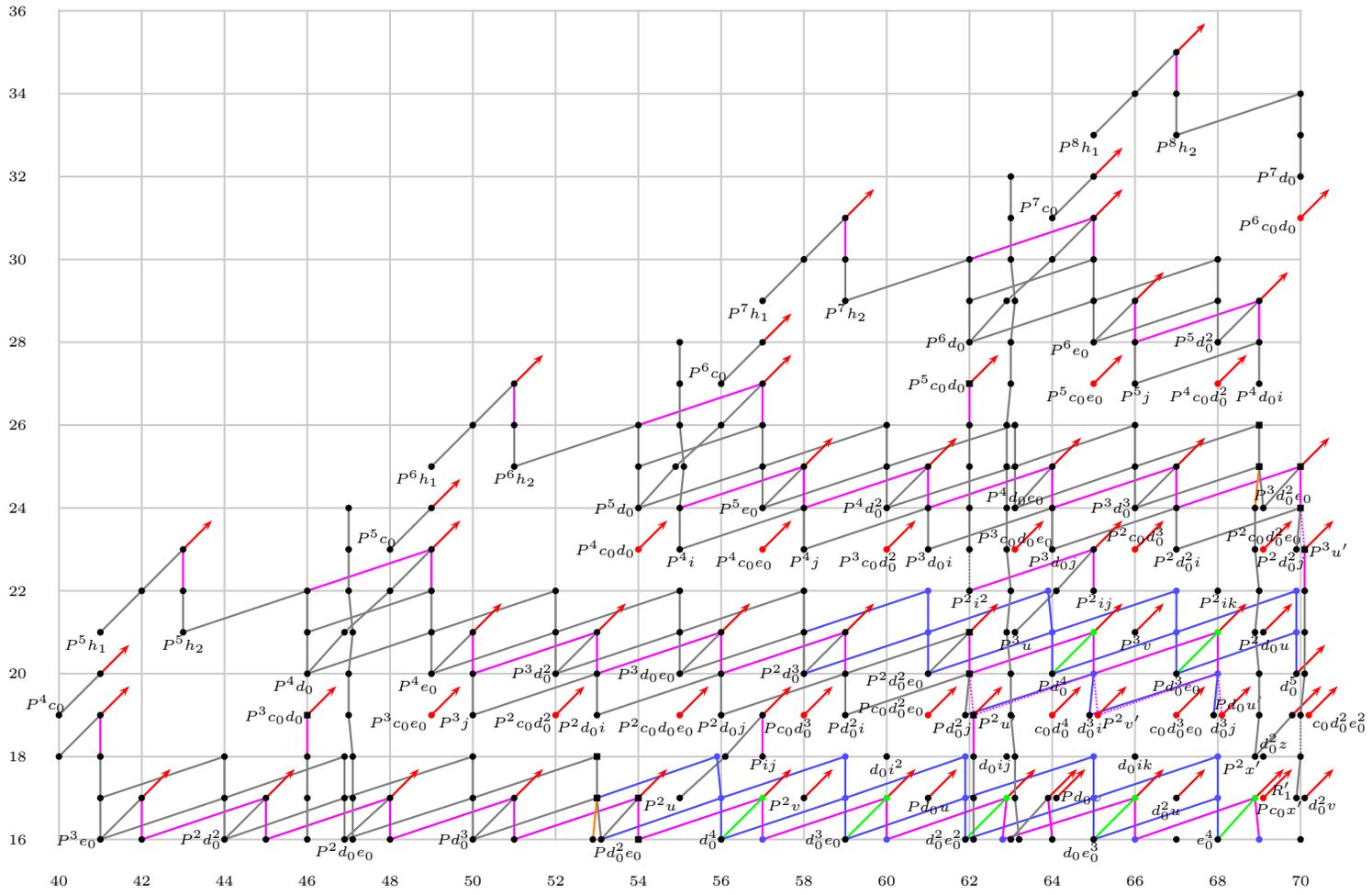
\begin{landscape}


\psset{linewidth=0.3mm}
\renewcommand{\cirrad}{0.08}

\psset{unit=0.57cm}



\end{landscape}

\bibliographystyle{amsalpha}

\begin{thebibliography}{SGA4}

\bibitem[Bl]{Bl}
S.\ Bloch,
\emph{Algebraic cycles and higher $K$-theory},
Adv.\ Math.\ \textbf{61} (1986) 267-Ð304. 

\bibitem[Br1]{Br1}
R.\ R.\ Bruner, 
\emph{The cohomology of the mod 2 Steenrod algebra: A computer calculation}, Wayne State University Research Report \textbf{37} (1997).

\bibitem[Br2]{Br2}
R.\ R.\ Bruner,
\emph{Squaring operations in $\Ext_\mathcal{A}(F_2, F_2)$}, preprint.

\bibitem[DHI]{DHI}
D.\ Dugger, M.\ A.\ Hill, and D.\ C.\ Isaksen,
\textit{The cohomology of the motivic Steenrod algebra over $\R$},
in preparation.

\bibitem[DI]{DI} 
D.\ Dugger and D.\ C.\ Isaksen,
\textit{The motivic Adams spectral sequence},
Geom.\ Topol.\ \textbf{14} (2010) 967-Ð1014. 

\bibitem[I1]{I1}
D.\ C.\ Isaksen,
\emph{The cohomology of motivic $A(2)$},
Homology, Homotopy Appl.\ \textbf{11} (2009) 251-Ð274. 

\bibitem[I2]{I2}
D.\ C.\ Isaksen,
\emph{Some differentials in the motivic Adams spectral sequence 
over $\Spec \C$}, 
in preparation.

\bibitem[MT]{MT}
M.\ Mahowald and M.\ Tangora,
\emph{An infinite subalgebra of $\Ext_A(Z_2,Z_2)$},
Trans.\ Amer.\ Math.\ Soc.\ \textbf{132} (1968) 263Ð-274. 

\bibitem[Ma1]{Ma1}
J.\ P.\ May, 
\emph{The cohomology of restricted Lie algebras and of Hopf algebras;
application to the Steenrod algebra}, Ph.D. dissertation, Princeton, 1964.

\bibitem[Ma2]{Ma2}
J.\ P.\ May, 
\emph{A general algebraic approach to Steenrod operations},
in The Steenrod algebra and its applications (Columbus, 1970), 153--231,
Lecture Notes in Mathematics \textbf{168}, Springer, 1970.

\bibitem[Mi]{Mi}
J.\ R.\ Milgram, 
\emph{Steenrod squares and higher Massey products},
Bol.\ Soc.\ Mat.\ Mexicana \textbf{13} (1968) 32-Ð57. 

\bibitem[T]{T} 
M.\ C.\ Tangora,
\emph{On the cohomology of the Steenrod algebra},
Math.\ Z.\ \textbf{116} (1970) 18-Ð64.  
 
\bibitem[V1]{V1}
V.\ Voevodsky,
\emph{Motivic cohomology with $\Z/2$-coefficients},
Publ.\ Math.\ Inst.\ Hautes \'Etudes Sci.\  \textbf{98} (2003) 59--104. 

\bibitem[V2]{V2}
V.\ Voevodsky,
\emph{Reduced power operations in motivic cohomology},
Publ.\ Math.\ Inst.\ Hautes \'Etudes Sci.\  \textbf{98} (2003) 1--57. 

\bibitem[V3]{V3}
V.\ Voevodsky,
\emph{Motivic Eilenberg-Mac Lane spaces}, 
Publ.\ Math.\ Inst.\ Hautes \'Etudes Sci.\ \textbf{112} (2010) 1-Ð99.


\end{thebibliography}

\end{document}